\documentclass[11pt]{article}
\usepackage{amsfonts,amscd,amssymb, amsmath}
\newtheorem{definition}{Definition}[section]
\newtheorem{lemma}[definition]{Lemma}
\newtheorem{proposition}[definition]{Proposition}
\newtheorem{corollary}[definition]{Corollary}
\newtheorem{remark}[definition]{Remark}

\newtheorem{theorem}[definition]{Theorem}
\newtheorem{example}[definition]{Example}

\def\mb{\mathbb}

\def\rawo\lonra{\longrightarrow}

\def\ot{\otimes}

\allowdisplaybreaks[2]

\newenvironment{proof}{{\it Proof.}}{\hfill $ \square $ \vskip 4mm}
\begin{document}
\title{Pseudosymmetric braidings, twines and twisted algebras} 
\author{Florin Panaite\thanks {Research
carried out while the first author was visiting the University of Antwerp, 
supported by a postdoctoral fellowship  
offered by FWO (Flemish Scientific Research Foundation). This author was 
also partially supported by the programme CEEX of the Romanian 
Ministry of Education and
Research, contract nr. 2-CEx06-11-20/2006.}\\
Institute of Mathematics of the 
Romanian Academy\\ 
PO-Box 1-764, RO-014700 Bucharest, Romania\\
e-mail: Florin.Panaite@imar.ro
\and 
Mihai D. Staic\thanks{Permanent address: Institute of Mathematics of the  
Romanian Academy, 
PO-Box 1-764, RO-014700 Bucharest, Romania.}\\
Department of Mathematics, Indiana University,\\
Rawles Hall, Bloomington, IN 47405, USA\\
e-mail: mstaic@indiana.edu 
\and
Freddy Van Oystaeyen\\
Department of Mathematics and Computer Science\\
University of Antwerp, Middelheimlaan 1\\
B-2020 Antwerp, Belgium\\
e-mail: Francine.Schoeters@ua.ac.be}
\date{}
\maketitle
\begin{abstract}
A {\em laycle} is the categorical analogue of a lazy cocycle.    
{\em Twines} (as introduced by Brugui\`{e}res) and {\em strong twines} (as   
introduced by the authors) are laycles satisfying some extra conditions. 
If $c$ is a braiding, the double braiding $c^2$ is always a twine; we  
prove that it is a strong twine if and only if $c$ satisfies a sort of 
modified braid relation (we call such $c$  
{\it pseudosymmetric}, as any symmetric braiding satisfies this relation).  
It is known that symmetric Yetter-Drinfeld categories are   
trivial; we prove that the Yetter-Drinfeld category $_H{\cal YD}^H$ over 
a Hopf algebra $H$ is pseudosymmetric if and only if $H$ is commutative and  
cocommutative. We introduce as well the Hopf algebraic counterpart of 
pseudosymmetric braidings under the name {\it pseudotriangular structures} 
and prove that all quasitriangular structures on the $2^{n+1}$-dimensional 
pointed Hopf algebras $E(n)$ are pseudotriangular.  
We observe that a laycle on a monoidal 
category induces a so-called {\it pseudotwistor} on every algebra in the 
category, and we obtain some general results (and give some examples) 
concerning pseudotwistors, inspired by properties of laycles and twines.  
\end{abstract}
\newpage
\section*{Introduction}
${\;\;\;\;\;}$ 
The notion of {\em symmetric category} is a classical concept in category 
theory. It consists of a monoidal category  
${\cal C}$ equipped with a family of natural isomorphisms 
$c_{X, Y}:X\otimes Y\rightarrow Y\otimes X$ satisfying natural 
``bilinearity'' conditions together with the symmetry relation  
$c_{Y, X}\circ c_{X, Y}=id_{X\otimes Y}$, for all $X, Y\in {\cal C}$. 
In 1985 Joyal and Street were led by natural considerations to drop this 
symmetry condition from the axioms, thus arriving at the concept of  
{\em braiding}, which afterwards became of central importance for the   
then emerging theory of quantum groups; for instance, if $(H, R)$ is a  
quasitriangular Hopf algebra as defined by Drinfeld, then the monoidal 
category $_H{\cal M}$ of left $H$-modules acquires a braiding defined by   
$R$, which is symmetric if and only if $R$ is triangular, i.e. 
$R_{21}R=1\otimes 1$. 

There exist many examples of symmetric braidings, as well as many 
examples of braidings which are not symmetric. Although some of the most 
basic examples of monoidal categories (such as the category of 
vector spaces) are symmetric, the symmetry condition is a rather 
restrictive requirement, a claim which is probably best illustrated by the 
following result of Pareigis (cf. \cite{pareigis}): if $H$ is a 
Hopf algebra, then the Yetter-Drinfeld category $_H{\cal YD}^H$ 
is symmetric if and only if $H$ is trivial (i.e. $H=k$). Thus, the most 
basic examples of braided categories arising in Hopf algebra theory 
are virtually never symmetric. 

It appears thus natural to look for braidings satisfying some 
generalized (or weakened) symmetry conditions. In a recent paper 
\cite{etingof}, Etingof and Gelaki proposed the concept of 
{\em quasisymmetric braiding}, as being a braiding with the property that 
$c_{Y, X}\circ c_{X, Y}=id_{X\otimes Y}$ for all $X$, $Y$ {\em simple}  
objects in the category, and classified quasisymmetric braided categories 
of exponential growth, generalizing Deligne's classification of 
symmetric categories of exponential growth. On the other hand, at the 
Hopf algebraic level, Liu and Zhu proposed in \cite{liuzhu} the concept 
of {\em almost-triangular} Hopf algebra, as being a quasitriangular 
Hopf algebra $(H, R)$ such that $R_{21}R$ is central in $H\otimes H$ 
(obviously, this concept generalizes the one of triangular Hopf algebra, but  
it is not clear whether it has a categorical counterpart).

The original aim of the present paper was to continue the study of some 
categorical concepts recently introduced in \cite{doru}, 
\cite{brug}, \cite{psvo} under the names {\em pure-braided structure}, 
{\em twine} and {\em strong twine}. We recall from \cite{brug} that a 
twine on a monoidal category ${\cal C}$ is a family of natural isomorphisms 
$D_{X, Y}:X\otimes Y\rightarrow X\otimes Y$ in ${\cal C}$ satisfying 
a certain list of axioms chosen in such a way that, if $c$ is a 
braiding on ${\cal C}$, then the so-called {\em double braiding} 
$c^2$ defined by $c^2_{X, Y}=c_{Y, X}\circ c_{X, Y}$ is a twine (by 
\cite{psvo}, the concept of twine is equivalent to the concept of 
pure-braided structure introduced in \cite{doru}). Moreover, twines are 
related to the pure braid groups in the same way in which braidings are 
related to the braid groups. A strong twine, as defined in \cite{psvo}, 
is also a family of natural isomorphisms 
$D_{X, Y}:X\otimes Y\rightarrow X\otimes Y$ in ${\cal C}$ satisfying a 
list of (easier looking) axioms, which imply the axioms of a twine. 
A double braiding $c^2$ is {\em not} always a strong twine, so we were 
led naturally to ask for what kind of braidings $c$ is $c^2$ a 
strong twine. The answer is that this happens if and only if  
$c$ satisfies the following condition:        
\begin{eqnarray*}
&&(c_{Y, Z}\otimes id_X)\circ (id_Y\otimes c_{Z, X}^{-1})\circ   
(c_{X, Y}\otimes id_Z)
=(id_Z\otimes c_{X, Y})\circ (c_{Z, X}^{-1}\otimes id_Y)
\circ (id_X\otimes c_{Y, Z}) 
\end{eqnarray*}
for all $X, Y, Z\in {\cal C}$. This is a sort of modified braid relation, 
and it is obvious that if $c$ is a symmetry then this condition becomes 
exactly the braid relation satisfied by any braiding; thus, any symmetric 
braiding satisfies the above relation, so what we obtained is a  
generalized symmetry condition. A braiding satisfying the above 
modified braid relation will be called {\em pseudosymmetric}. 
It should be emphasized that, although we arrived at this concept in an 
indirect way (via double braidings and strong twines), the pseudosymmetry 
relation does not depend on these concepts and could have  
been introduced directly. Anyway, this concept is supported and  
further justified by our main result: if $H$ is a Hopf algebra (with 
bijective antipode) then the canonical braiding of the Yetter-Drinfeld 
category $_H{\cal YD}^H$ is pseudosymmetric if and only if $H$ is 
commutative and cocommutative. In view of Pareigis' result mentioned above, 
this shows that pseudosymmetries are far more numerous than symmetries; 
and in the opposite direction, it shows that {\em not} every braiding is 
pseudosymmetric (this was not so obvious a priori). 
Note also that, incidentally, our theorem provides a  
characterization of commutative and cocommutative Hopf algebras 
solely in terms of their Yetter-Drinfeld categories. 

We introduce the Hopf algebraic counterpart of pseudosymmetric  
braidings, under the name {\em pseudotriangular structure}, as being a 
quasitriangular structure $R$ on a Hopf algebra $H$ satisfying the 
modified quantum Yang-Baxter equation $R_{12}R_{31}^{-1}R_{23}=
R_{23}R_{31}^{-1}R_{12}$ (from which it is visible that triangular 
implies pseudotriangular) or equivalently the element $F=R_{21}R$ 
satisfies the condition $F_{12}F_{23}=F_{23}F_{12}$, which shows 
immediately that almost-triangular implies pseudotriangular. We analyze 
in detail a class of quasitriangular Hopf algebras, namely the 
$2^{n+1}$-dimensional pointed Hopf algebras $E(n)$ whose 
quasitriangular structures and cleft extensions have been classified in 
\cite{panvo} and \cite{panvan}: we prove that all quasitriangular 
structures of $E(n)$ (which are in bijection with $n\times n$ matrices) 
are pseudotriangular, and the only almost-triangular structures of 
$E(n)$ are the triangular ones (which are in bijection with symmetric 
$n\times n$ matrices); in particular, this shows that pseudotriangular  
does {\em not} imply almost-triangular.

Apart from leading us to consider a certain class of braidings (the 
pseudosymmetric ones), the study of twines led us also to  
consider certain classes of {\em pseudotwistors}, as introduced in 
\cite{lpvo}. In order to explain this, we need to introduce first some 
terminology. A basic object we use all over the paper is a monoidal 
structure of the identity functor on a monoidal category (for instance, 
this is part of the axioms for twines and strong twines). We needed to 
have a name for such an object, and in order to choose it we relied 
on the fact that these objects are the categorical analogues of 
{\em lazy cocycles}, a concept recently introduced in Hopf algebra  
theory and studied in a series of papers (\cite{bichon}, \cite{c}, 
\cite{cc}, \cite{chen}, \cite{cp}, \cite{sch}). Thus, we have chosen 
the name {\bf laycle}, as derived from {\bf la}zy coc{\bf ycle}. These 
laycles have some properties similar to those of lazy cocycles, for 
instance they act by conjugation on braidings and it is possible to 
define for them an analogue of the Hopf lazy cohomology.             

The concept of pseudotwistor (with particular cases called {\em twistor} 
and {\em braided twistor}) was introduced in \cite{lpvo} as an  
abstract and axiomatic device for ``twisting'' the multiplication of 
an algebra in a monoidal category in order to obtain a new algebra structure 
(on the same object). More precisely, if $(A, \mu , u)$ is an  
algebra in a monoidal category ${\cal C}$, a pseudotwistor for $A$ is a 
morphism $T:A\otimes A\rightarrow A\otimes A$ in ${\cal C}$, for which 
there exist two morphisms $\tilde{T}_1, \tilde{T}_2:A\otimes A\otimes A
\rightarrow A\otimes A\otimes A$ in ${\cal C}$, called the companions of $T$, 
satisfying a list of axioms ensuring that $(A, \mu \circ T, u)$ is 
also an algebra in ${\cal C}$. Examples of pseudotwistors are abundant, 
cf. \cite{lpvo}. For instance, if $c$ is a braiding on ${\cal C}$, 
then $c^2_{A, A}$ is a pseudotwistor for every algebra $A$ in ${\cal C}$. 
Since a double braiding is in particular a twine, this raises the 
natural question whether any twine induces a pseudotwistor on every 
algebra in the category. It turns out that something more general holds, 
namely that any laycle has this property. 
This seems to show that pseudotwistors are 
``local'' versions of laycles (in the same sense in which twisting maps 
are ``local'' versions of braidings, see \cite{jlpvo} for the meaning 
of these concepts and references), but this is not quite true,  
because for instance a composition of laycles is a laycle while a 
composition of pseudotwistors is not in general a pseudotwistor. We 
introduce thus the concept of {\em strong pseudotwistor}, as a better 
candidate for being a local version of laycles (for instance, 
a composition of a strong pseudotwistor with itself is again a 
strong pseudotwistor). We also introduce a sort of local  
version of twines, under the name {\em pure pseudotwistor}, as being 
a pseudotwistor whose companions satisfy the condition 
$(\tilde{T}_2\otimes id)\circ (id\otimes \tilde{T}_1)=(id \otimes 
\tilde{T}_1)\circ (\tilde{T}_2\otimes id)$. Quite interestingly, it turns 
out that virtually all the concrete examples of pseudotwistors 
we are aware of are pure.

What we discussed above are basically facts about pseudotwistors inspired 
by properties of laycles and twines. In the last section of the paper 
we complete the picture of the interplay between laycles and twines, 
on the one hand, and pseudotwistors, on the other hand, by presenting 
a result in the opposite direction. Namely, inspired by a result in 
\cite{lpvo} concerning pseudotwistors and twisting maps, we prove that, 
if ${\cal C}$ is a monoidal category, $T$ a laycle and $d$ a braiding 
on ${\cal C}$ related in a certain way, then the families 
$d'_{X, Y}=d_{X, Y}\circ T_{X, Y}$ and $d''_{X, Y}=T_{Y, X}\circ d_{X, Y}$ 
are also braidings on ${\cal C}$. We prove also a sort of converse result, 
leading thus to a characterization of generalized double braidings 
(i.e. twines of the type $c'_{Y, X}\circ c_{X, Y}$, with $c$, $c'$ 
braidings).           
\section{Preliminaries}\label{sec1}
\setcounter{equation}{0}
${\;\;\;\;\;}$
In this section we recall basic definitions and results and we fix   
notation to be used throughout the paper. 
All algebras, linear 
spaces, etc, will be over a base field $k$; unadorned $\ot$ means $\ot _k$. 
All monoidal categories are assumed to be strict, with unit denoted by $I$.  
For a Hopf algebra $H$ with comultiplication $\Delta$ we  
denote $\Delta (h)=h_1\ot h_2$, for all $h\in H$. 
Unless otherwise stated, $H$ will denote a Hopf algebra with
bijective antipode $S$. For terminology concerning Hopf algebras and  
monoidal categories we refer to \cite{k}, \cite{m}. 

A linear map $\sigma :H\ot H\rightarrow k$ is called a {\bf left 
2-cocycle} if it satisfies the condition
\begin{eqnarray}
&&\sigma (a_1, b_1)\sigma (a_2b_2, c)=\sigma (b_1, c_1)\sigma (a, 
b_2c_2), \label{leftco}
\end{eqnarray}
for all $a, b, c\in H$, and it is called a {\bf right  
2-cocycle} if it satisfies the condition 
\begin{eqnarray}
&&\sigma (a_1b_1, 
c)\sigma (a_2, b_2)=\sigma (a, b_1c_1)\sigma (b_2, c_2). \label{rightco}
\end{eqnarray}

Given a linear map $\sigma :H\ot H\rightarrow k$, define a product 
$\cdot_{\sigma }$ on $H$ by 
$h\cdot _{\sigma }h'=\sigma (h_1, h'_1)h_2h'_2$, for all $h, h'\in H$. 
Then $\cdot_{\sigma }$ is associative if and only if $\sigma $ is
a left 2-cocycle. If we define $\cdot _{\sigma }$ by 
$h\cdot _{\sigma }h'=h_1h'_1\sigma (h_2, h'_2)$, 
then $\cdot _{\sigma }$ is associative if and only if $\sigma $
is a right 2-cocycle. In any of the two cases, $\sigma$ is
normalized (i.e. $\sigma (1, h)= \sigma (h, 1)=\varepsilon (h)$
for all $h\in H$) if and only if $1_H$ is the unit for $\cdot
_{\sigma }$.  If $\sigma $ is a normalized left (respectively
right) 2-cocycle, we denote the algebra $(H, \cdot _{\sigma })$
by $_{\sigma }H$ (respectively $H_{\sigma }$). It is well-known
that $_{\sigma }H$ (respectively $H_{\sigma }$) is a right
(respectively left) $H$-comodule algebra via the comultiplication
$\Delta $ of $H$. If $\sigma :H\ot H\rightarrow k$ is normalized
and convolution invertible, then $\sigma $ is a left 2-cocycle if
and only if $\sigma ^{-1}$ is a right 2-cocycle. 

If $\gamma :H\rightarrow k$ is linear, normalized (i.e. $\gamma
(1)=1$) and convolution invertible, define
\begin{eqnarray*}
&&D^1(\gamma ):H\ot H\rightarrow k, \;\;\;
D^1(\gamma )(h, h')=\gamma (h_1)\gamma (h'_1)\gamma ^{-1}(h_2h'_2),
\;\;\;\;\;\forall \;h, h'\in H.
\end{eqnarray*}
Then $D^1(\gamma )$ is a normalized and convolution invertible
left 2-cocycle.

We recall from \cite{bichon} some facts about lazy cocycles and  
lazy cohomology. The set $Reg^1 (H)$ (respectively $Reg^2 (H)$)
consisting of normalized and convolution invertible linear maps
$\gamma :H\rightarrow k$ (respectively $\sigma :H\ot H\rightarrow
k$), is a group with respect to the convolution product. 
An element $\gamma 
\in Reg^1 (H)$ is called {\bf lazy} if 
$\gamma (h_1)h_2=h_1\gamma (h_2)$, for all $h\in H$.  
The set of lazy elements of $Reg^1 (H)$, denoted by $Reg^1_L (H)$,
is a central subgroup of $Reg^1 (H)$. An element $\sigma \in
Reg^2 (H)$ is called {\bf lazy} if
\begin{eqnarray}
&&\sigma (h_1, h'_1)h_2h'_2=h_1h'_1\sigma (h_2, h'_2),\;\;\;\;\;
\forall \;h, h'\in H. \label{lazy2}
\end{eqnarray}
The set of lazy elements of $Reg^2 (H)$, denoted by $Reg^2_L (H)$,
is a subgroup of $Reg^2 (H)$. We denote by $Z^2 (H)$ the set of
left 2-cocycles on $H$ and by $Z^2_L (H)$ the set $Z^2 (H)\cap
Reg^2_L (H)$ of normalized and convolution invertible lazy
2-cocycles. If $\sigma \in Z^2_L(H)$, then the algebras $_{\sigma
}H$ and $H_{\sigma }$ coincide and will be denoted by $H(\sigma
)$; moreover, $H(\sigma )$ is an $H$-bicomodule algebra via
$\Delta $. 

It is well-known that in general the set $Z^2 (H)$ of left
2-cocycles is not closed under convolution. One of the main
features of lazy 2-cocycles is that the set $Z^2_L (H)$ is closed
under convolution, and that the convolution inverse of an element
$\sigma \in Z^2_L (H)$ is again a lazy 2-cocycle, so $Z^2_L (H)$
is a group under convolution. In particular, a lazy 2-cocycle is
also a right 2-cocycle.
 
Consider now the map 
$D^1:Reg^1 (H)\rightarrow Reg^2 (H)$, $D^1(\gamma )(h, h')=\gamma  
(h_1)\gamma (h'_1)\gamma ^{-1}(h_2h'_2)$, 
for all $h, h'\in H$. 
Then, by \cite{bichon}, the map $D^1$ induces a group morphism 
$Reg^1_L (H)\rightarrow Z^2_L (H)$, with image contained in 
the centre of $Z^2_L (H)$; denote by $B^2_L (H)$ this central
subgroup $D^1(Reg^1_L (H))$ of $Z^2_L (H)$ (its elements are
called {\bf lazy 2-coboundaries}). Then define the {\bf second  
lazy cohomology group} $H^2_L (H)=Z^2_L (H)/B^2_L (H)$.

Dually, an invertible element $T\in H\otimes H$ is called a {\bf lazy  
twist} if
\begin{eqnarray*}
&&(\varepsilon \otimes id)(T)=1=(id\otimes \varepsilon)(T), 
\label{t1} \\
&&(id\otimes \Delta)(T)(1\otimes T)=(\Delta \otimes id)(T)(T\otimes 1), 
\label{t2} \\
&&\Delta (h) T=T \Delta(h),  \;\;\;\; \forall \; h\in H.
\label{t5}
\end{eqnarray*}
As a consequence of these axioms we also have 
$(1\otimes T)(id\otimes \Delta)(T)=(T\otimes 1)(\Delta \otimes id)(T)$. 
One can define the analogues of $Z^2_L(H)$, $B^2_L(H)$ and 
$H^2_L(H)$ with lazy twists instead of lazy cocycles; these will be denoted 
respectively by $Z^2_{LT}(H)$, $B^2_{LT}(H)$ and $H^2_{LT}(H)$.  
\begin{remark}
If ${\cal C}$ is a monoidal category and $T_{X, Y}:X\ot Y\rightarrow X\ot Y$  
is a family of natural isomorphisms in ${\cal C}$, the naturality of 
$T$ implies (for all $X, Y, Z\in {\cal C}$): 
\begin{eqnarray}
&&(T_{X, Y}\ot id_Z)\circ T_{X\ot Y, Z}=T_{X\ot Y, Z}\circ 
(T_{X, Y}\ot id_Z),  \label{lac1} \\
&&(id_X\ot T_{Y, Z})\circ T_{X, Y\ot Z}=T_{X, Y\ot Z}\circ 
(id_X\ot T_{Y, Z}).  \label{lac2}
\end{eqnarray}
\end{remark}
\begin{definition} (\cite{k})
Let ${\cal C}=({\cal C}, \otimes , I)$ and 
${\cal D}=({\cal D}, \otimes , I)$ be monoidal categories. A 
{\bf monoidal functor} from ${\cal C}$ to ${\cal D}$ is a triple 
$(F, \varphi _0, \varphi _2)$ where $F:{\cal C}\rightarrow {\cal D}$ is a 
functor, $\varphi _0$ is an isomorphism in ${\cal D}$ from $I$ to 
$F(I)$ and $\varphi _2(U, V):F(U)\otimes F(V)\rightarrow F(U\otimes V)$ is 
a family of natural isomorphisms in ${\cal D}$ indexed by all couples 
$(U, V)$ of objects in ${\cal C}$ such that, for all $U, V, W\in {\cal C}$: 
\begin{eqnarray*}
\varphi _2(U\otimes V, W)\circ (\varphi _2(U, V)\otimes id_{F(W)})&=&
\varphi _2(U, V\otimes W)\circ (id_{F(U)}\otimes \varphi _2(V, W)), \\
\varphi _2(I, U)\circ (\varphi _0\otimes id_{F(U)})&=&id_{F(U)}, \\
\varphi _2(U, I)\circ (id_{F(U)}\otimes \varphi _0)&=&id_{F(U)}.
\end{eqnarray*}
\end{definition}
\begin{definition} (\cite{joyalstreet}) 
Let ${\cal C}$ be a monoidal category. A {\bf braiding} on ${\cal C}$  
consists of a family of natural isomorphisms $c_{X, Y}:X\otimes Y
\rightarrow Y\otimes X$ in ${\cal C}$ such that, 
for all $X, Y, Z\in {\cal C}$:
\begin{eqnarray}
&&c_{X, Y\otimes Z}=(id_Y\otimes c_{X, Z})\circ (c_{X, Y}\otimes id_Z), 
\label{braid1} \\
&&c_{X\otimes Y, Z}=(c_{X, Z}\otimes id_Y)\circ (id_X\otimes c_{Y, Z}).
\label{braid2}
\end{eqnarray}
As consequences of the axioms we also have $c_{X, I}=c_{I, X}=id_X$ and 
the braid relation 
\begin{multline}
\;\;\;\;\;\;\;\;\;\;\;\;\;\;\;\;\;\;\;
(c_{Y, Z}\otimes id_X)\circ (id_Y\otimes c_{X, Z})\circ 
(c_{X, Y}\otimes id_Z)\\
=(id_Z\otimes c_{X, Y})\circ (c_{X, Z}\otimes id_Y)
\circ (id_X\otimes c_{Y, Z}). \label{braideq}
\end{multline}
If moreover $c$ satisfies $c_{Y, X}\circ c_{X, Y}=id_{X\otimes Y}$, 
for all $X, Y\in {\cal C}$, then $c$ is called a {\bf symmetry}.  
\end{definition} 
\begin{definition} (\cite{drinfeld}) Let $H$ be a Hopf algebra. An invertible 
element $R\in H\otimes H$ is called a {\bf quasitriangular structure} 
for $H$ if 
\begin{eqnarray*}
&&(\Delta \otimes id)(R)=R_{13}R_{23}, \\
&&(id\otimes \Delta )(R)=R_{13}R_{12}, \\
&&(\varepsilon \otimes id)(R)=(id\otimes \varepsilon )(R)=1, \\
&&\Delta ^{cop}(h)R=R\Delta (h), \;\;\;\forall \;h\in H. 
\end{eqnarray*}
If moreover $R$ satisfies $R_{21}R=1\otimes 1$, then $R$ is called 
{\bf triangular}. If $R$ is a quasitriangular (respectively triangular) 
structure for $H$, then the monoidal category $_H{\cal M}$ of left 
$H$-modules becomes braided (respectively symmetric), with braiding 
given by $c_{M, N}:M\otimes N\rightarrow N\otimes M$, 
$c_{M, N}(m\otimes n)=R^2\cdot n\otimes R^1\cdot m$, for all 
$M, N\in $$\;_H{\cal M}$, $m\in M$, $n\in N$. 
\end{definition}
\begin{definition} (\cite{doru}) \label{purebraided}
Let ${\cal C}$  be a monoidal category. 
A {\bf pure-braided} structure on ${\cal C}$ consists of  
two families of natural   
isomorphisms $A_{U,V,W}:U\otimes V\otimes W \to U\otimes V\otimes W$ 
and $B_{U,V,W}:U\otimes V\otimes W \to 
U\otimes V\otimes W$ in ${\cal C}$ such that (for all 
$U, V, W, X\in {\cal C}$):
\begin{eqnarray}
&&A_{U\otimes V,W,X}=A_{U,V\otimes W,X}\circ (id_U\otimes A_{V,W,X}),  
\label{a1} \\
&&A_{U,V,W\otimes X}=(A_{U,V,W}\otimes id_X)\circ A_{U,V\otimes W,X}, 
\label{a2} \\
&&B_{U\otimes V,W,X}=(id_U \otimes B_{V,W,X})\circ B_{U,V\otimes W,X}, 
\label{baba} \\
&&B_{U,V,W\otimes X}=B_{U,V\otimes W,X}\circ (B_{U,V,W}\otimes id_X),  
\label{b2} \\
&&(A_{U,V,W}\otimes id_X)\circ (id_U \otimes B_{V,W,X})=
(id_U \otimes B_{V,W,X})\circ (A_{U,V,W}\otimes id_X), 
\label{cab} \\
&&A_{U,I,V}=B_{U,I,V}.  
\label{t1t} 
\end{eqnarray}
A category equipped with a pure-braided structure is called a pure-braided  
category. 
\end{definition}
\begin{definition} (\cite{brug}) Let ${\cal C}$ be  a monoidal category.  
A {\bf twine} on ${\cal C}$ is a family of natural isomorphisms   
$D_{X,Y}:X\otimes Y\to X\otimes Y$ in ${\cal C}$  
satisfying the axioms (for all $X, Y, Z, W\in {\cal C}$):
\begin{eqnarray}
D_{I,I}&=&id_I, \label{twine1}\\
(D_{X,Y}\otimes id_Z)\circ D_{X\otimes Y, Z}&=&
(id_X \otimes D_{Y,Z})\circ D_{X,Y\otimes Z}, 
\label{twine2}
\end{eqnarray}
\begin{multline}
\;\;\;\;\;\;\;\;(D_{X\otimes Y,Z}\otimes id_W)\circ 
(id_X\otimes D_{Y,Z}^{-1}\otimes id_W)\circ (id_X\otimes D_{Y,Z\otimes W})\\
=(id_X\otimes D_{Y,Z\otimes W})\circ  
(id_X\otimes D_{Y,Z}^{-1}\otimes id_W)\circ (D_{X\otimes Y,Z}\otimes id_W). 
\label{twine3}
\end{multline}
A category equipped with a twine is called an entwined category. 
If $({\cal C}, D)$ is entwined then we also have    
$D_{X,I}=D_{I,X}=id_X$, for all $X\in {\cal C}$. 
\end{definition}

By \cite{psvo}, these two concepts are equivalent in a certain 
(precise) sense.
\begin{proposition} (\cite{brug}) 
Let ${\cal C}$ be a monoidal category and $c$, $c'$ braidings on 
${\cal C}$. Then the family $T_{X, Y}:=c'_{Y, X}\circ c_{X, Y}$ is a twine, 
called a {\bf generalized double braiding}; if $c=c'$ the family 
$c_{Y, X}\circ c_{X, Y}$ is called a {\bf double braiding}.
\end{proposition}
\begin{definition} (\cite{psvo}) Let ${\cal C}$ be a monoidal category  
and $T_{X, Y}:X\otimes Y\to X\otimes Y$ a family of natural isomorphisms  
in ${\cal C}$.   
We say that $T$ is a {\bf strong twine} (or $({\cal C}, T)$ is  
strongly entwined) if for all $X, Y, Z\in {\cal C}$ we have:  
\begin{eqnarray}
&&T_{I,I}=id_I, \label{str1} \\
&&(T_{X, Y}\otimes id_Z)\circ T_{X\otimes Y, Z}=
(id_X\otimes T_{Y, Z})\circ T_{X, Y\otimes Z}, \label{str2} \\
&&(T_{X, Y}\otimes id_Z)\circ (id_X\otimes T_{Y, Z})=
(id_X\otimes T_{Y, Z})\circ (T_{X, Y}\otimes id_Z). \label{str3}
\end{eqnarray}
\end{definition}
\begin{proposition}\label{lilu} (\cite{psvo}) 
If $({\cal C}, T)$ is strongly entwined then   
$({\cal C}, T)$ is entwined.
\end{proposition}
\begin{proposition} (\cite{borcherds1}, \cite{borcherds2}) \label{borch} 
Let $A$ be an algebra with multiplication denoted by   
$\mu _A=\mu $ and let $T:A\otimes A\rightarrow A\otimes A$ be a linear map  
satisfying the   
following conditions: $T(1\otimes a)=1\otimes a$, $T(a\otimes 1)=
a\otimes 1$, for all $a\in A$, and 
\begin{eqnarray}
&&\mu _{23}\circ T_{12}\circ T_{13}=
T\circ \mu _{23}:A\otimes A\otimes A\rightarrow A\otimes A, \label{rmat1}\\
&&\mu _{12}\circ T_{23}\circ T_{13}=
T\circ \mu _{12}:A\otimes A\otimes A\rightarrow A\otimes A, \label{rmat2}\\
&&T_{12}\circ T_{13}\circ T_{23}=
T_{23}\circ T_{13}\circ T_{12}:A\otimes A\otimes A\rightarrow 
A\otimes A\otimes A, 
\label{rmat3} 
\end{eqnarray}
with standard notation for $\mu _{ij}$ and $T_{ij}$. Then the 
map $\mu \circ T:A\otimes A\rightarrow A$ defines an associative    
algebra   
structure on $A$, with the same unit 1. The map $T$ is called an 
{\bf $R$-matrix} for $A$.  
\end{proposition} 
\section{Laycles and quasi-braidings} \label{seclaycles}
\setcounter{equation}{0}
\begin{definition}
Let ${\cal C}$ be  a monoidal category and 
$T_{X,Y}:X\otimes Y\to X\otimes Y$ a family of natural isomorphisms in 
${\cal C}$. We say that $T$ is a {\bf laycle} if for all  
$X, Y, Z\in {\cal C}$ we have:
\begin{eqnarray}
&&T_{I,I}=id_I, \label{c0}\\
&&(T_{X,Y}\otimes id_Z)\circ T_{X\otimes Y, Z}=
(id_X \otimes T_{Y,Z})\circ T_{X,Y\otimes Z}.
\label{c1}
\end{eqnarray}
A category equipped with a laycle is called a laycled category.
\end{definition}
\begin{remark}
It $T$ is a laycle on ${\cal C}$ 
then we also have    
$T_{X,I}=T_{I,X}=id_X$, for all $X\in {\cal C}$. Also, it is clear that 
if $({\cal C},T)$ is entwined then $({\cal C},T)$ is laycled.
\end{remark}
\begin{remark}
It is obvious that $T$ is a laycle if and only if $(id_{{\cal C}}, id_I, 
\varphi _2(X, Y):=T_{X, Y})$ is a monoidal functor from ${\cal C}$ to 
itself. So, directly from the properties of monoidal functors, it follows 
that the composition of two laycles is a laycle and the inverse of a laycle 
is a laycle. 
\end{remark}
\begin{example} {\em  
Let $H$ be a Hopf algebra, $\sigma \in Reg^2_L(H)$  
and ${\cal C}={\cal M}^H$, the category of right $H$-comodules, 
with tensor product $(m\ot n)_{(0)}\ot (m\ot n)_{(1)}=
(m_{(0)}\ot n_{(0)})\ot m_{(1)}n_{(1)}$. Define 
$T_{M, N}(m\ot n)=m_{(0)}\ot n_{(0)}\sigma (m_{(1)}, n_{(1)})$, for all 
$M, N\in {\cal M}^H$, $m\in M$, $n\in N$.  
Then $\sigma $ is a lazy 
2-cocycle on $H$ if and only if $T$ is a laycle on ${\cal M}^H$. Dually, 
if $F=F^1\otimes F^2\in H\otimes H$ is invertible and satisfies  
$(\varepsilon \otimes id)(F)=(id\otimes \varepsilon )(F)=1$, consider the 
category $_H{\cal M}$ of left $H$-modules, with tensor product given by 
$h\cdot (m\otimes n)=h_1\cdot m\otimes h_2\cdot n$, for all 
$M, N\in $$\;_H{\cal M}$, $m\in M$, $n\in N$; define $T_{M, N}(m\otimes n)=
F^1\cdot m\otimes F^2\cdot n$.  Then $F$ is a lazy twist if and only if 
$T$ is a laycle on $_H{\cal M}$.} 
\end{example}

If $T$ is a laycle on ${\cal C}$, we define  
the families $T^b_{X, Y, Z},\; T^f_{X, Y, Z}:X\otimes Y\otimes Z 
\rightarrow X\otimes Y\otimes Z$ (notation as in \cite{brug}) 
of natural isomorphisms in ${\cal C}$ 
associated to it, by 
\begin{eqnarray}
&&T^b_{X,Y,Z}:=(id_X \otimes T_{Y,Z}^{-1})\circ T_{X\otimes Y, Z}=
T_{X,Y\otimes Z}\circ (T_{X, Y}^{-1}\otimes id_Z), \label{ela} \\
&&T^f_{X,Y,Z}:=T_{X\otimes Y, Z}\circ (id_X \otimes T_{Y,Z}^{-1})=
(T_{X, Y}^{-1}\otimes id_Z)\circ T_{X,Y\otimes Z}. \label{fla} 
\end{eqnarray} 
\begin{proposition}\label{caractlaycle} 
Let ${\cal C}$ be a monoidal category.\\
(i) If $T$ is a laycle on ${\cal C}$ then for all $U, V, W\in {\cal C}$ 
we have 
\begin{eqnarray}
&&T^f_{U\otimes V,W,X}=T^f_{U,V\otimes W,X}\circ (id_U\otimes T^f_{V,W,X}),  
\label{impl11}\\
&&T^f_{U,V,W\otimes X}=(T^f_{U,V,W}\otimes id_X)\circ T^f_{U,V\otimes W,X}.
\label{impl12}
\end{eqnarray}
Conversely, if $A_{U, V, W}:U\otimes V\otimes W\rightarrow 
U\otimes V\otimes W$ is a family of natural isomorphisms such that  
(\ref{impl11}) and (\ref{impl12}) with $A$ instead of $T^f$ hold, 
then $T_{U,V}:=A_{U,I,V}$ is a laycle on ${\cal C}$.\\
(ii) If $T$ is a laycle on ${\cal C}$ then 
for all $U, V, W\in {\cal C}$ we have   
\begin{eqnarray}
&&T^b_{U\otimes V,W,X}=(id_U \otimes T^b_{V,W,X})\circ T^b_{U,V\otimes W,X}, 
\label{impl21} \\
&&T^b_{U,V,W\otimes X}=T^b_{U,V\otimes W,X}\circ (T^b_{U,V,W}\otimes id_X). 
\label{impl22}
\end{eqnarray}
Conversely, if $B_{U, V, W}:U\otimes V\otimes W\rightarrow 
U\otimes V\otimes W$ is a family of natural isomorphisms such that  
(\ref{impl21}) and (\ref{impl22}) with $B$ instead of $T^b$ hold, 
then $T_{U,V}:=B_{U,I,V}$ is a laycle on ${\cal C}$.
\end{proposition}
\begin{proof} We prove (i), while (ii) is similar and left to the reader. 
We compute:
\begin{eqnarray*}
T^f_{U\otimes V,W,X}&=&
T_{U\otimes V\otimes W, X}\circ (id_U\otimes id_V\otimes T_{W,X}^{-1})\\
&=&T_{U\otimes V\otimes W, X}\circ(id_{U}\otimes T_{V\otimes W,X}^{-1})\\
&&\circ 
(id_{U}\otimes T_{V\otimes W,X})\circ 
(id_U\otimes id_V\otimes T_{W,X}^{-1})\\
&\overset{(\ref{c1})}{=}& 
T^f_{U, V\otimes W,X}\circ(id_U\otimes T^f_{V,W,X}), 
\end{eqnarray*}
proving (\ref{impl11}); the proof of (\ref{impl12}) is similar and left to 
the reader.
 
Assume now that $A_{-,-,-}$ is a family of natural isomorphisms  
satisfying (\ref{impl11}) and (\ref{impl12}); then obviously the family  
$T_{U,V}=A_{U,I,V}$ consists also of natural isomorphisms.  
If in (\ref{impl11}) we take $V=W=X=I$ we obtain $T_{U,I}=T_{U,I}\circ 
(id_U\otimes T_{I,I})$, 
hence $T_{I,I}=id_I$. If we take $W=I$ in (\ref{impl11}) and $V=I$ in 
(\ref{impl12}) we obtain 
\begin{eqnarray*}
&&T_{U\otimes V, X}=A_{U,V,X}\circ (id_U\otimes T_{V,X}), \;\;\;\;\;\;  
T_{U,W\otimes X}=(T_{U,W}\otimes id_X)\circ A_{U,W,X}, 
\end{eqnarray*}
which together imply (\ref{c1}). 
\end{proof} 

The categorical analogue of the operator $D^1$ from the Preliminaries 
looks as follows:
\begin{proposition} (\cite{brug}) Let ${\cal C}$ be a monoidal category and  
$R_X:X\rightarrow X$ a family of natural isomorphisms in ${\cal C}$ 
such that $R_I=id_I$. Then the family 
\begin{eqnarray}
&&D^1(R)_{X, Y}:=(R_X\ot R_Y)\circ 
R_{X\ot Y}^{-1}=R_{X\ot Y}^{-1}\circ (R_X\ot R_Y) \label{D1R}
\end{eqnarray}  
is a laycle on ${\cal C}$. 
\end{proposition}

The next result (whose proof is straightforward and will be omitted) 
provides the categorical analogue of Hopf lazy cohomology:
\begin{proposition} \label{lazyco}
Let ${\cal C}$ be a small monoidal category. Then:\\
(i) If we denote by $Reg^1_L({\cal C})$ the set of families of 
natural isomorphisms $R_X:X\rightarrow X$ in ${\cal C}$ such that 
$R_I=id_I$, then  
$Reg^1_L({\cal C})$ is an abelian group under composition. \\
(ii) The set of  
laycles on ${\cal C}$ is a group, denoted by $Z^2_L({\cal C})$.\\
(iii) The map $D^1:Reg^1_L({\cal C})\rightarrow Z^2_L({\cal C})$ is a 
group morphism with image (denoted by $B^2_L({\cal C})$) contained in the  
centre of $Z^2_L({\cal C})$.

We denote by $H^2_L({\cal C})$ the group $Z^2_L({\cal C})/B^2_L({\cal C})$,   
and call it the lazy cohomology of ${\cal C}$.  
\end{proposition}

A basic property of lazy cocycles on Hopf algebras (see \cite{bichon}) is  
that they act on coquasitriangular structures. This property extends 
to the categorical setting:
\begin{proposition} \label{conjug} 
Let ${\cal C}$ be a monoidal category, $T$ a laycle and $c$ a  
braiding on ${\cal C}$. Then the family $c^T_{X, Y}:= 
T_{Y, X}\circ c_{X, Y}\circ T_{X, Y}^{-1}$ is also a braiding on ${\cal C}$. 
\end{proposition}
\begin{proof}
The naturality of $c$    
with respect to the morphisms $id_X$ and $T_{Y, Z}^{-1}$ 
together with (\ref{braid1}) imply 
\begin{multline}\label{ajut1}
\;\;\;\;\;\;\;\;\;\;\;\;\;\;\;\;\;\;\;\;\;(T_{Y, Z}^{-1}\otimes id_X)\circ   
(id_Y\otimes c_{X, Z})\circ (c_{X, Y}\otimes id_Z)\\
=(id_Y\otimes c_{X, Z})\circ (c_{X, Y}\otimes id_Z)\circ     
(id_X\otimes T_{Y, Z}^{-1}). 
\end{multline}
The naturality of $c$ with respect to the   
morphisms $T_{X, Y}^{-1}$ and $id_Z$ together with  
(\ref{braid2}) imply 
\begin{multline}\label{ajut2}
\;\;\;\;\;\;\;\;\;\;\;\;\;\;\;\;\;\;\;(id_Z\otimes T_{X, Y}^{-1})\circ 
(c_{X, Z}\otimes id_Y)\circ (id_X\otimes c_{Y, Z})\\
=(c_{X, Z}\otimes id_Y)\circ    
(id_X\otimes c_{Y, Z})\circ (T_{X, Y}^{-1}\otimes id_Z).   
\end{multline}
We check (\ref{braid1}) for $c^T$; we compute:
\begin{eqnarray*}
c^T_{X, Y\otimes Z}&=&T_{Y\otimes Z, X}\circ c_{X, Y\otimes Z}\circ 
T_{X, Y\otimes Z}^{-1}\\
&\overset{(\ref{braid1}),\; (\ref{c1})}{=}& 
(id_Y\otimes T_{Z, X})\circ T_{Y, Z\otimes X}\circ (T_{Y, Z}^{-1}
\otimes id_X)\circ (id_Y\otimes c_{X, Z})\\
&&\circ (c_{X, Y}\otimes id_Z)
\circ (id_X\otimes T_{Y, Z})\circ T_{X\otimes Y, Z}^{-1}\circ 
(T_{X, Y}^{-1}\otimes id_Z)\\
&\overset{(\ref{ajut1})}{=}& 
(id_Y\otimes T_{Z, X})\circ T_{Y, Z\otimes X}\circ (id_Y\otimes c_{X, Z})\\
&&\circ (c_{X, Y}\otimes id_Z)\circ T_{X\otimes Y, Z}^{-1}\circ  
(T_{X, Y}^{-1}\otimes id_Z)\\
&\overset{naturality\;of\;T}{=}& 
(id_Y\otimes T_{Z, X})\circ (id_Y\otimes c_{X, Z})\circ T_{Y, X\otimes Z}\\
&&\circ T_{Y\otimes X, Z}^{-1}\circ (c_{X, Y}\otimes id_Z)\circ   
(T_{X, Y}^{-1}\otimes id_Z)\\
&\overset{(\ref{c1})}{=}& 
(id_Y\otimes T_{Z, X})\circ (id_Y\otimes c_{X, Z})\circ (id_Y\otimes 
T_{X, Z}^{-1})\\
&&\circ (T_{Y, X}\otimes id_Z)\circ (c_{X, Y}\otimes id_Z)\circ   
(T_{X, Y}^{-1}\otimes id_Z)\\
&=&(id_Y\otimes c^T_{X, Z})\circ (c^T_{X, Y}\otimes id_Z), \;\;\;q.e.d.
\end{eqnarray*}
Similarly, we check (\ref{braid2}) for $c^T$:
\begin{eqnarray*}
c^T_{X\otimes Y, Z}&=&T_{Z, X\otimes Y}\circ c_{X\otimes Y, Z}\circ 
T_{X\otimes Y, Z}^{-1}\\
&\overset{(\ref{braid2}),\; (\ref{c1})}{=}& 
(T_{Z, X}\otimes id_Y)\circ T_{Z\otimes X, Y}\circ (id_Z\otimes 
T_{X, Y}^{-1})\circ (c_{X, Z}\otimes id_Y)\\
&&\circ (id_X\otimes c_{Y, Z})\circ (T_{X, Y}\otimes id_Z)\circ 
T_{X, Y\otimes Z}^{-1}\circ (id_X\otimes T_{Y, Z}^{-1})\\  
&\overset{(\ref{ajut2})}{=}& 
(T_{Z, X}\otimes id_Y)\circ T_{Z\otimes X, Y}\circ (c_{X, Z}\otimes id_Y)\\
&&\circ (id_X\otimes c_{Y, Z})\circ  
T_{X, Y\otimes Z}^{-1}\circ (id_X\otimes T_{Y, Z}^{-1})\\  
&\overset{naturality\;of\;T}{=}& 
(T_{Z, X}\otimes id_Y)\circ (c_{X, Z}\otimes id_Y)\circ T_{X\otimes Z, Y}\\
&&\circ T_{X, Z\otimes Y}^{-1}\circ (id_X\otimes c_{Y, Z}) 
\circ (id_X\otimes T_{Y, Z}^{-1})\\  
&\overset{(\ref{c1})}{=}& 
(T_{Z, X}\otimes id_Y)\circ (c_{X, Z}\otimes id_Y)\circ 
(T_{X, Z}^{-1}\otimes id_Y)\\
&&\circ (id_X\otimes T_{Z, Y})\circ (id_X\otimes c_{Y, Z}) 
\circ (id_X\otimes T_{Y, Z}^{-1})\\
&=&(c^T_{X, Z}\otimes id_Y)\circ (id_X\otimes c^T_{Y, Z}),   
\end{eqnarray*}
finishing the proof.
\end{proof}
\begin{proposition}\label{trivact}
Let ${\cal C}$ be a monoidal category, $c$ a braiding on ${\cal C}$ 
and $R_X:X\rightarrow X$ a family of natural isomorphisms in ${\cal C}$ 
such that $R_I=id_I$. Then $c^{D^1(R)}=c$, where $D^1(R)$ is the 
laycle given by (\ref{D1R}). 
\end{proposition}
\begin{proof}
Follows immediately by using the naturality of $c$ and $R$. 
\end{proof}
\begin{corollary}
If ${\cal C}$ is a small monoidal category, then the group $H^2_L({\cal C})$  
acts on the set of braidings of ${\cal C}$.  
\end{corollary}
\begin{proposition}\label{equival}
In the hypotheses of Proposition \ref{conjug}, the braided monoidal 
categories $({\cal C}, c)$ and $({\cal C}, c^T)$ are equivalent (as 
braided monoidal categories).  
\end{proposition}
\begin{proof}
We define the monoidal functor $(F, \varphi _0, \varphi _2):{\cal C}
\rightarrow {\cal C}$ by $F=id_{\cal C}$, $\varphi _0=id_I$ and 
$\varphi _2(X, Y):X\otimes Y\rightarrow X\otimes Y$, 
$\varphi _2(X, Y):= T_{X, Y}^{-1}$, which is obviously a   
monoidal equivalence. Moreover, the formula $c^T_{X, Y}=T_{Y, X}\circ 
c_{X, Y}\circ T_{X, Y}^{-1}$ expresses exactly the fact that 
$(F, \varphi _0, \varphi _2)$ is a braided functor from $({\cal C}, c)$ 
to $({\cal C}, c^T)$ 
\end{proof}

If ${\cal C}$ is a braided monoidal category with braiding $c$, we denote 
by $Br ({\cal C}, c)$ its Brauer group as introduced in \cite{voz}. 
Thus, as a consequence of Proposition \ref{equival}, we obtain the following 
generalization of \cite{carnovale}, Proposition 3.1:
\begin{corollary}
In the hypotheses of Proposition \ref{conjug}, the Brauer groups 
$Br ({\cal C}, c)$ and $Br ({\cal C}, c^T)$ are isomorphic. 
\end{corollary}

\begin{definition}
Let ${\cal C}$ be a monoidal category. A {\bf quasi-braiding} on  
${\cal C}$ 
is a family of natural isomorphisms $q_{X,Y}:X\otimes Y\rightarrow 
Y\otimes X$ in ${\cal C}$ satisfying the following axioms (for all   
$X, Y, Z\in {\cal C}$):
\begin{eqnarray}
&&q_{I, I}=id_I, \label{unitatea} \\
&&q_{X,Z\otimes Y}\circ (id_X\otimes q_{Y,Z})=
q_{Y\otimes X,Z}\circ (q_{X,Y}\otimes id_Z). \label{quasico}
\end{eqnarray}
If $q_{Y, X}\circ q_{X, Y}=id_{X\otimes Y}$ for all  
$X, Y\in {\cal C}$, then $({\cal C}, q)$ is what Drinfeld calls a 
{\bf coboundary category} in \cite{drin}.    
\end{definition}

\begin{remark} If $q$ is a quasi-braiding on ${\cal C}$ then we also have  
$q_{X, I}=q_{I, X}=id_X$ and 
\begin{multline}
\;\;\;\;\;\;\;\;(q_{Y,Z}\otimes id_X)\circ q_{X,Y\otimes Z}=
q_{X,Z\otimes Y}\circ (id_X\otimes q_{Y,Z})\\
=q_{Y\otimes X,Z}\circ   
(q_{X,Y}\otimes id_Z)=(id_Z\otimes q_{X,Y})\circ q_{X\otimes Y,Z}. 
\label{conseqquasi}
\end{multline}
Consequently, the family $p_{X, Y}:=q_{Y, X}^{-1}$ is also a 
quasi-braiding.   
\end{remark}

The concept of quasi-braiding was considered (with a different name) by 
L. M. Ionescu in \cite{ionescu}, as follows. Define a monoidal category 
${\cal C}_{op}$, which is the same as ${\cal C}$ as a category, has the 
same unit $I$, and reversed tensor product: $X\otimes _{op}Y=Y\otimes X$. 
Then, a family $q_{X, Y}:X\otimes Y\rightarrow Y\otimes X$ is a 
quasi-braiding on ${\cal C}$ if and only if $(id_{{\cal C}}, id_I, 
\varphi _2(X, Y):=q_{X, Y})$ is a monoidal functor from ${\cal C}_{op}$ to  
${\cal C}$, or equivalently  
$(id_{{\cal C}}, id_I,  
\varphi _2(X, Y):=q_{Y, X})$ is a monoidal functor from ${\cal C}$ to 
${\cal C}_{op}$.   
As noted in \cite{ionescu}, any braiding is a quasi-braiding (this 
follows easily by (\ref{braid1}), (\ref{braid2}) and 
(\ref{braideq})), and quasi-braidings are related to Drinfeld's coboundary 
Hopf algebras:  
\begin{definition}(\cite{drinfeld})
A coboundary Hopf algebra is a pair $(H, R)$, where $H$ is a Hopf algebra and 
$R\in H\otimes H$ is an invertible element such that:
\begin{eqnarray}
&&R_{12}(\Delta \otimes id)(R)=R_{23}(id\otimes \Delta )(R), \label{cob1}\\
&&(\varepsilon \otimes id)(R)=(id\otimes \varepsilon )(R)=1, \label{cob2}\\
&&\Delta ^{cop}(h)R=R\Delta (h), \;\;\;\forall \;\;h\in H, \label{cob3}\\
&&R_{21}R=1\otimes 1. \label{cob4}
\end{eqnarray}
If $R$ does not necessarily satisfy (\ref{cob4}), we call it a 
{\bf quasi-coboundary}. 
\end{definition}
\begin{proposition}(\cite{ionescu})
Let $H$ be a Hopf algebra and $R=R^1\otimes R^2\in H\otimes H$ an 
invertible element. If $U, V$ are left $H$-modules, define   
$q_{U,V}:U\otimes V\rightarrow V\otimes U$ by   
$q_{U, V}(u\otimes v)=R^2\cdot v\otimes R^1\cdot u$. 
Then $q$ is a quasi-braiding  
on $_H{\cal M}$ if and only if $R$ is a quasi-coboundary on $H$.   
\end{proposition}
\begin{remark}
If $T$ is a laycle on a monoidal category ${\cal C}$, then the family 
${\cal T}_{X, Y}:=T_{Y, X}$ is a laycle on ${\cal C}_{op}$. 
\end{remark}

From the description of laycles and quasi-braidings as monoidal 
structures for some identity functors and the fact that a composition of 
monoidal functors is monoidal, we obtain: 
\begin{proposition} \label{compmon}
Let ${\cal C}$ be a monoidal category, $T$ a laycle and $p$, $q$ two   
quasi-braidings on ${\cal C}$. Then the family $D_{X, Y}:=p_{Y, X}\circ    
q_{X, Y}$ is a laycle on ${\cal C}$ and 
the families $q'_{X, Y}:=T_{Y, X}\circ  
q_{X, Y}$ and $q''_{X, Y}:=q_{X, Y}\circ T_{X, Y}$ are  
quasi-braidings on ${\cal C}$. 
\end{proposition}
\begin{corollary}
Let ${\cal C}$ be a monoidal category, $T$ a laycle and $q$ a   
quasi-braiding on ${\cal C}$. Then the family $q^T_{X, Y}:= 
T_{Y, X}\circ q_{X, Y}\circ T_{X, Y}^{-1}$ is also a quasi-braiding  
on ${\cal C}$.  
\end{corollary}
\begin{remark}
Proposition \ref{trivact} is also true with quasi-braidings instead of  
braidings, so we obtain also an action of $H^2_L({\cal C})$ on the set of 
quasi-braidings of ${\cal C}$. 
\end{remark}
\begin{proposition}
Let ${\cal C}$ be a monoidal category, $c$ a braiding and $q$ a 
quasi-braiding on ${\cal C}$. Then the family $c^q_{X, Y}:=q_{Y, X}^{-1}
\circ c_{Y, X}\circ q_{X, Y}$ is also a braiding on ${\cal C}$. 
Moreover, the braided categories $({\cal C}, c)$ and $({\cal C}, c^q)$ are 
braided equivalent. 
\end{proposition}
\begin{proof}
Follows immediately from Proposition \ref{conjug}, since $c^q=c^T$, 
where $T$ is the laycle $T_{X, Y}=q_{X, Y}^{-1}\circ c_{X, Y}$.
\end{proof}
\begin{remark}
For the particular case when $q$ itself is a braiding, we will obtain an  
alternative proof in Proposition \ref{consteor}.
\end{remark}

Let now ${\cal C}$ be a small monoidal category. We denote by 
${\mb Z}^2({\cal C})$ the set of all natural isomorphisms in ${\cal C}$ 
that are laycles or quasi-braidings. Then, with notation as in 
Proposition \ref{lazyco}, we have:  
\begin{proposition} 
(i) ${\mb Z}^2({\cal C})$ is a group.\\
(ii) $Z_L^2({\cal C})$ is an index 2 subgroup in ${\mb Z}^2({\cal C})$.\\
(iii) $B_L^2({\cal C})$ is a central subgroup in ${\mb Z}^2({\cal C})$. 

We define the ``cohomology group''    
${\mb H}^2({\cal C}):={\mb Z}^2({\cal C})/B_L^2({\cal C})$. 
\end{proposition}
\begin{proof}
We give first the explicit description of the  multiplication in 
${\mb Z}^2({\cal C})$. Take $R$ and $P$ quasi-braidings,   
$S$ and $T$ laycles. We have, for all $U, V\in {\cal C}$: 
\begin{eqnarray*}
&&(ST)_{U,V}=S_{U,V}\circ T_{U,V}, \\ 
&&(TR)_{U,V}=T_{V,U}\circ R_{U,V}, \\
&&(RT)_{U,V}=R_{U,V}\circ T_{U,V}, \\
&&(RP)_{U,V}=R_{V,U}\circ P_{U,V}. 
\end{eqnarray*}
Now (i) and (ii) follow from Proposition \ref{compmon}, while (iii) 
is just an easy computation.
\end{proof}

Similarly, if $H$ is a Hopf algebra, we may consider the group 
${\mb Z}^2(H)$ consisting of the elements in $H\otimes H$ that are 
lazy twists or quasi-coboundaries, its central subgroup $B^2_{LT}(H)$  
and the ``cohomology group'' ${\mb H}^2(H)={\mb Z}^2(H)/B^2_{LT}(H)$.

\begin{example}{\em Let $k$ be a field with $char(k) \neq 2$ and  
$H=k[C_2]$, the group algebra of the cyclic group with 
two elements $C_2$ (denote its generator by $g$). One can see that the lazy  
twists on $H$ are given by the formula 
$T_a=\frac{3+a}{4}(1\otimes 1)+\frac{1-a}{4}(1\otimes g) 
+\frac{1-a}{4}(g\otimes 1)-\frac{1-a}{4}(g\otimes g)$, with 
$a\in k^*$. It is interesting to note that $T_0$ is not invertible but has 
all the other properties in the definition of a lazy twist. 

Consider the element 
$\theta_{\alpha}=\frac{1+g}{2}+\alpha\frac{1-g}{2}\in H$,    
with $\alpha \in k$. One can see that $\theta_{\alpha}$ is invertible 
if and only if $\alpha\neq 0$. Also it is easy to see that 
$T_{\alpha^{-2}}=\Delta(\theta_{\alpha})
(\theta_{\alpha}^{-1}\otimes \theta_{\alpha}^{-1})$ and so 
$T_a$ is trivial in ${\mb H}^2(H)$ if and only if $a\in (k^*)^2$. 
One can also note that $T_aT_b=T_{ab}$. 

Since $H$ is commutative and cocommutative, one can see that the 
quasi-coboundaries for $H$ are given by the formula 
$R_a=\frac{3+a}{4}(1\otimes 1)+\frac{1-a}{4}(1\otimes g) 
+\frac{1-a}{4}(g\otimes 1)-\frac{1-a}{4}(g\otimes g)$, with $a\in k^*$. 
Among these, only $R_1$ and $R_{-1}$ are quasitriangular. 
If we put everything together we obtain  
${\mb H}^2(H)=k^*/(k^*)^2\times C_2$.}
\end{example} 
\section{Strong twines and pseudosymmetric braidings}
\setcounter{equation}{0}
${\;\;\;\;\;}$
A key result for this section is the following characterization   
of strong twines:
\begin{proposition} \label{strongequal}
Let ${\cal C}$ be a monoidal category and $T$ a laycle on  
${\cal C}$. Then $T$ is a strong twine if and only if the families 
$T^b$ and $T^f$  
given by (\ref{ela}) and (\ref{fla}) coincide. 
\end{proposition}
\begin{proof}
Let $X, Y, Z\in {\cal C}$ and assume that $T$ is a strong twine; 
then we have: 
\begin{eqnarray*}
T^b_{X, Y, Z}&=&(id_X\otimes T_{Y, Z}^{-1})\circ T_{X\otimes Y, Z}\\
&=&(T_{X, Y}^{-1}\otimes id_Z)\circ (T_{X, Y}\otimes id_Z)\circ 
(id_X\otimes T_{Y, Z}^{-1})\circ T_{X\otimes Y, Z}\\
&\overset{(\ref{str3})}{=}&(T_{X, Y}^{-1}\otimes id_Z)\circ 
(id_X\otimes T_{Y, Z}^{-1})\circ (T_{X, Y}\otimes id_Z)\circ 
T_{X\otimes Y, Z}\\
&\overset{(\ref{str2})}{=}&(T_{X, Y}^{-1}\otimes id_Z)\circ 
(id_X\otimes T_{Y, Z}^{-1})\circ (id_X\otimes T_{Y, Z})\circ 
T_{X, Y\otimes Z}\\
&=&(T_{X, Y}^{-1}\otimes id_Z)\circ T_{X, Y\otimes Z}\\
&=&T^f_{X, Y, Z}.
\end{eqnarray*}
Conversely, assume that $T^b=T^f$. By using (\ref{ela}), (\ref{c1}) and 
(\ref{fla}) it is easy to see that $T^b_{X, Y, Z}\circ  
(T_{X, Y}\otimes id_Z)\circ (id_X\otimes T_{Y, Z})=T^f_{X, Y, Z}\circ 
(id_X\otimes T_{Y, Z})\circ (T_{X, Y}\otimes id_Z)$, and since $T^b=T^f$ 
it follows that (\ref{str3}) holds.  
\end{proof}
\begin{definition} (\cite{street})  
Let ${\cal C}$ be a monoidal category. A D-structure on 
${\cal C}$ consists of a family of natural morphisms $R_X:X\rightarrow X$  
in ${\cal C}$, such that $R_I=id_I$ and (for all $X, Y, Z\in {\cal C}$):
\begin{eqnarray}
&&(R_{X\ot Y}\ot id_Z)(id_X\ot R_{Y\ot Z})=(id_X\ot R_{Y\ot Z})
(R_{X\ot Y}\ot id_Z). \label{formD}
\end{eqnarray}
\end{definition}

It was proved in \cite{psvo} that if $R$ is a $D$-structure consisting of 
isomorphisms then the family $D^1(R)$ given by (\ref{D1R}) is a 
strong twine. Using Proposition \ref{strongequal} we can prove the 
converse:
\begin{proposition}
Let ${\cal C}$ be a monoidal category and $R_X:X\rightarrow X$ a family 
of natural isomorphisms in ${\cal C}$ with $R_I=id_I$. Then  
$D^1(R)$ is a strong twine if and only if $R$ is a $D$-structure. 
\end{proposition}
\begin{proof}
We compute:
\begin{eqnarray*}
D^1(R)^b_{X, Y, Z}&=&(id_X\otimes R_{Y\otimes Z})\circ 
(id_X\otimes R_Y^{-1}\otimes R_Z^{-1})\circ (R_{X\otimes Y}\otimes R_Z)\circ 
R_{X\otimes Y\otimes Z}^{-1}\\
&=&(id_X\otimes R_{Y\otimes Z})\circ 
(id_X\otimes R_Y^{-1}\otimes R_Z^{-1})\circ (id_X\otimes R_Y\otimes R_Z)\\
&&\circ (id_X\otimes R_Y^{-1}\otimes id_Z)
\circ (R_{X\otimes Y}\otimes id_Z)\circ R_{X\otimes Y\otimes Z}^{-1}\\
&=&(id_X\otimes R_{Y\otimes Z})
\circ (id_X\otimes R_Y^{-1}\otimes id_Z)
\circ (R_{X\otimes Y}\otimes id_Z)\circ R_{X\otimes Y\otimes Z}^{-1}\\
&\overset{naturality \;of\; R}{=}&
(id_X\otimes R_Y^{-1}\otimes id_Z)\circ (id_X\otimes R_{Y\otimes Z})
\circ (R_{X\otimes Y}\otimes id_Z)\circ R_{X\otimes Y\otimes Z}^{-1}, 
\end{eqnarray*}
and similarly one can see that 
\begin{eqnarray*}
D^1(R)^f_{X, Y, Z}&=&
(id_X\otimes R_Y^{-1}\otimes id_Z)\circ (R_{X\otimes Y}\otimes id_Z) 
\circ (id_X\otimes R_{Y\otimes Z})\circ R_{X\otimes Y\otimes Z}^{-1}, 
\end{eqnarray*}
and it is clear that $D^1(R)^b=D^1(R)^f$ (i.e. $D^1(R)$ is a strong twine) 
if and only if (\ref{formD}) holds. 
\end{proof}
 
We recall that a (generalized) double braiding is always a twine; 
it is natural to ask under what conditions is it a strong twine. The answer 
is provided by our next result: 
\begin{theorem}
Let ${\cal C}$ be a monoidal category, $c$ and $d$ braidings on 
${\cal C}$ and $T_{X, Y}=d_{Y, X}\circ c_{X, Y}$. Then $T$ is a strong twine 
if and only if the following relation holds, for all $X, Y, Z\in {\cal C}$:
\begin{multline}
(d_{Z, X}\otimes id_Y)\circ (id_Z\otimes c_{X, Y}^{-1})\circ  
(c_{Y, Z}\otimes id_X)\\
\circ (d_{Z, Y}\otimes id_X)\circ (id_Z\otimes c_{X, Y})\circ  
(c_{X, Z}\otimes id_Y)\\
=(id_X\otimes c_{Y, Z})\circ (d_{X, Y}^{-1}\otimes id_Z)
\circ (id_Y\otimes d_{Z, X})\\
\circ (id_Y\otimes c_{X, Z})\circ (d_{X, Y}\otimes id_Z)
\circ (id_X\otimes d_{Z, Y}). \label{braidtare}
\end{multline}
\end{theorem}
\begin{proof}
We compute the families $T^b$ and $T^f$:
\begin{eqnarray*}
T^b_{X, Y, Z}&=&(id_X\otimes T_{Y, Z}^{-1})\circ T_{X\otimes Y, Z}\\
&\overset{(\ref{braid1})}{=}&
(id_X\otimes c_{Y, Z}^{-1})\circ (id_X\otimes d_{Z, Y}^{-1})\circ 
(id_X\otimes d_{Z, Y})\circ (d_{Z, X}\otimes id_Y)\circ c_{X\otimes Y, Z}\\
&=&(id_X\otimes c_{Y, Z}^{-1})\circ (d_{Z, X}\otimes id_Y)
\circ c_{X\otimes Y, Z},
\end{eqnarray*}
\begin{eqnarray*}
T^f_{X, Y, Z}&=&T_{X\otimes Y, Z}\circ (id_X\otimes T_{Y, Z}^{-1})\\
&\overset{(\ref{braid2})}{=}&
d_{Z, X\otimes Y}\circ (c_{X, Z}\otimes id_Y)\circ (id_X\otimes c_{Y, Z})
\circ (id_X\otimes c_{Y, Z}^{-1})\circ (id_X\otimes d_{Z, Y}^{-1})\\
&=&d_{Z, X\otimes Y}\circ (c_{X, Z}\otimes id_Y)
\circ (id_X\otimes d_{Z, Y}^{-1}).
\end{eqnarray*}
By Proposition \ref{strongequal}, $T$ is a strong twine if  
and only if $T^b=T^f$, and this holds if and only if 
\begin{eqnarray}
&&(d_{Z, X}\otimes id_Y)\circ c_{X\otimes Y, Z}\circ (id_X\otimes d_{Z, Y})
=(id_X\otimes c_{Y, Z})\circ d_{Z, X\otimes Y}\circ (c_{X, Z}\otimes id_Y).
\label{TBTF}
\end{eqnarray}
Thus, it is enough to prove that the left hand sides of equations 
(\ref{braidtare}) and (\ref{TBTF}) coincide, and the same for the right 
hand sides. We compute:\\[2mm]
${\;\;\;\;\;\;\;}$
$(d_{Z, X}\otimes id_Y)\circ c_{X\otimes Y, Z}\circ (id_X\otimes d_{Z, Y})$
\begin{eqnarray*}
&=&
(d_{Z, X}\otimes id_Y)\circ (id_Z\otimes c_{X, Y}^{-1})
\circ (c_{Y, Z}\otimes id_X)
\circ c_{X, Y\otimes Z}\circ (id_X\otimes d_{Z, Y})\\
&=&
(d_{Z, X}\otimes id_Y)\circ (id_Z\otimes c_{X, Y}^{-1})
\circ (c_{Y, Z}\otimes id_X)\circ (d_{Z, Y}\otimes id_X)
\circ c_{X, Z\otimes Y}\\
&=&
(d_{Z, X}\otimes id_Y)\circ (id_Z\otimes c_{X, Y}^{-1})\circ  
(c_{Y, Z}\otimes id_X)\\
&&\circ (d_{Z, Y}\otimes id_X)\circ (id_Z\otimes c_{X, Y})\circ  
(c_{X, Z}\otimes id_Y)
\end{eqnarray*} 
(for the first equality we used (\ref{conseqquasi}), for the second 
the naturality of $c$ and for the third (\ref{braid1})),\\[2mm] 
${\;\;\;\;\;\;\;}$
$(id_X\otimes c_{Y, Z})\circ d_{Z, X\otimes Y}\circ (c_{X, Z}\otimes id_Y)$
\begin{eqnarray*}
&=&(id_X\otimes c_{Y, Z})\circ (d_{X, Y}^{-1}\otimes id_Z)\circ 
(id_Y\otimes d_{Z, X})
\circ d_{Z\otimes X, Y}\circ (c_{X, Z}\otimes id_Y)\\
&=&(id_X\otimes c_{Y, Z})\circ (d_{X, Y}^{-1}\otimes id_Z)\circ 
(id_Y\otimes d_{Z, X})\circ (id_Y\otimes c_{X, Z})\circ 
d_{X\otimes Z, Y}\\
&=&(id_X\otimes c_{Y, Z})\circ (d_{X, Y}^{-1}\otimes id_Z)
\circ (id_Y\otimes d_{Z, X})\\
&&\circ (id_Y\otimes c_{X, Z})\circ (d_{X, Y}\otimes id_Z)
\circ (id_X\otimes d_{Z, Y})
\end{eqnarray*}
(for the first equality we used (\ref{conseqquasi}), for the second the 
naturality of $d$ and for the third (\ref{braid2})), finishing the proof. 
\end{proof}
\begin{definition}
Let ${\cal C}$ be a monoidal category and $c$ a braiding on ${\cal C}$. 
We will say that $c$ is a {\bf pseudosymmetry} if the following 
condition holds, for all $X, Y, Z\in {\cal C}$:
\begin{multline}
\;\;\;\;\;\;\;\;\;\;\;\;\;\;\;\;\;\;\;
(c_{Y, Z}\otimes id_X)\circ (id_Y\otimes c_{Z, X}^{-1})\circ  
(c_{X, Y}\otimes id_Z)\\
=(id_Z\otimes c_{X, Y})\circ (c_{Z, X}^{-1}\otimes id_Y)
\circ (id_X\otimes c_{Y, Z}). \label{pseudosym}
\end{multline}
In this case we will say that ${\cal C}$ is  
a {\bf pseudosymmetric braided category}.  
\end{definition}

If $c$ is a symmetry, i.e. $c_{Z, X}^{-1}=c_{X, Z}$, then obviously $c$ is a 
pseudosymmetry, by (\ref{braideq}). 
\begin{theorem} \label{carpseudosym}
Let ${\cal C}$ be a monoidal category and $c$ a braiding on ${\cal C}$. 
Then the double braiding $T_{X, Y}=c_{Y, X}\circ c_{X, Y}$ is a strong 
twine if and only if $c$ is a pseudosymmetry.  
\end{theorem}
\begin{proof}
In (\ref{braidtare}) written for $c=d$ we have, by (\ref{braideq}), 
\begin{eqnarray*}
(c_{Z, Y}\otimes id_X)\circ (id_Z\otimes c_{X, Y})\circ  
(c_{X, Z}\otimes id_Y)=
(id_Y\otimes c_{X, Z})\circ (c_{X, Y}\otimes id_Z)
\circ (id_X\otimes c_{Z, Y}),  
\end{eqnarray*}
so (\ref{braidtare}) reduces in this case to (\ref{pseudosym}). 
\end{proof}

Let $H$ be a Hopf algebra. Consider the category $_H{\cal YD}^H$ of 
left-right Yetter-Drinfeld modules over $H$, whose objects are vector 
spaces $M$ that are left $H$-modules (denote the action by 
$h\otimes m\mapsto h\cdot m$) and right $H$-comodules (denote the coaction 
by $m\mapsto m_{(0)}\otimes m_{(1)}\in M\otimes H$) satisfying the 
compatibility condition 
\begin{eqnarray}
&&(h\cdot m)_{(0)}\otimes (h\cdot m)_{(1)}=h_2\cdot m_{(0)}\otimes 
h_3m_{(1)}S^{-1}(h_1), \;\;\;\forall \;\;h\in H, \;m\in M. \label{yd}
\end{eqnarray}
It is a monoidal category, with tensor product given by 
\begin{eqnarray*}
&&h\cdot (m\otimes n)=h_1\cdot m\otimes h_2\cdot n, \;\;\;
(m\otimes n)_{(0)}\otimes (m\otimes n)_{(1)}=m_{(0)}\otimes n_{(0)}
\otimes n_{(1)}m_{(1)}.
\end{eqnarray*}
Moreover, it has a (canonical) braiding given by 
\begin{eqnarray*}
&&c_{M, N}:M\otimes N\rightarrow N\otimes M, \;\;\;c_{M, N}(m\otimes n)=
n_{(0)}\otimes n_{(1)}\cdot m, \\
&&c_{M, N}^{-1}:N\otimes M\rightarrow M\otimes N, \;\;\;
c_{M, N}^{-1}(n\otimes m)=S(n_{(1)})\cdot m\otimes n_{(0)}. 
\end{eqnarray*}

It is known (cf. \cite{pareigis}) that this braiding is a symmetry 
only in the degenerate case $H=k$. 
\begin{theorem} \label{co}
The canonical braiding of $_H{\cal YD}^H$ is pseudosymmetric if and only if 
$H$ is commutative and cocommutative. 
\end{theorem} 
\begin{proof}
Assume first that $H$ is commutative and cocommutative; in this case, 
the compatibility condition (\ref{yd}) becomes the Long condition 
\begin{eqnarray}
&&(h\cdot m)_{(0)}\otimes (h\cdot m)_{(1)}=h\cdot m_{(0)}\otimes 
m_{(1)}, \;\;\;\forall \;\;h\in H, \;m\in M. \label{long}
\end{eqnarray}
For all $X, Y, Z\in $$\;_H{\cal YD}^H$, $x\in X$, $y\in Y$, $z\in Z$ 
we compute:\\[2mm]
${\;\;\;\;\;\;\;}$
$(c_{Y, Z}\otimes id_X)\circ (id_Y\otimes c_{Z, X}^{-1})\circ  
(c_{X, Y}\otimes id_Z)(x\otimes y\otimes z)$
\begin{eqnarray*}
&=&(c_{Y, Z}\otimes id_X)\circ (id_Y\otimes c_{Z, X}^{-1})(y_{(0)}\otimes 
y_{(1)}\cdot x\otimes z)\\
&=&(c_{Y, Z}\otimes id_X)(y_{(0)}\otimes 
S((y_{(1)}\cdot x)_{(1)})\cdot z\otimes (y_{(1)}\cdot x)_{(0)})\\
&\overset{(\ref{long})}{=}&
(c_{Y, Z}\otimes id_X)(y_{(0)}\otimes 
S(x_{(1)})\cdot z\otimes y_{(1)}\cdot x_{(0)})\\
&=&(S(x_{(1)})\cdot z)_{(0)}\otimes (S(x_{(1)})\cdot z)_{(1)}\cdot y_{(0)}
\otimes y_{(1)}\cdot x_{(0)}\\
&\overset{(\ref{long})}{=}&
S(x_{(1)})\cdot z_{(0)}\otimes z_{(1)}\cdot y_{(0)}\otimes 
y_{(1)}\cdot x_{(0)}\\
&\overset{(\ref{long})}{=}&
S(x_{(1)})\cdot z_{(0)}\otimes (z_{(1)}\cdot y)_{(0)}\otimes 
(z_{(1)}\cdot y)_{(1)}\cdot x_{(0)}\\
&=&(id_Z\otimes c_{X, Y})(S(x_{(1)})\cdot z_{(0)}\otimes x_{(0)}\otimes 
z_{(1)}\cdot y)\\
&=&(id_Z\otimes c_{X, Y})\circ (c_{Z, X}^{-1}\otimes id_Y)
(x\otimes z_{(0)}\otimes z_{(1)}\cdot y)\\
&=&(id_Z\otimes c_{X, Y})\circ (c_{Z, X}^{-1}\otimes id_Y)\circ (id_X\otimes 
c_{Y, Z})(x\otimes y\otimes z),
\end{eqnarray*}
proving that $c$ is pseudosymmetric. 

Conversely, assume that $c$ is pseudosymmetric. We consider the two usual   
Yetter-Drinfeld structures on the vector space $H$: the first one, 
denoted by $H_1$, is $H$ with the usual (regular) left module structure 
and with comodule structure $\rho _1(h)=h_2\otimes h_3S^{-1}(h_1)$, 
and the second, denoted by $H_2$, is $H$ with    
module structure given by $h\cdot g=h_2gS^{-1}(h_1)$ and comodule structure 
$\rho _2(h)=h_1\otimes h_2$.

We prove first that $H$ is cocommutative. Let $h\in H$; we will  
apply the pseudosymmetry condition (\ref{pseudosym}) for  
$X=H_1$, $Y=H_2$, $Z=H_1$ on the element $1\otimes h\otimes 1$:\\[2mm] 
${\;\;\;\;\;\;}$
$(c_{Y,Z}\otimes id_X)\circ (id_Y\otimes c_{Z,X}^{-1})\circ 
(c_{X,Y}\otimes id_Z) (1\otimes h\otimes 1)$
\begin{eqnarray*}
&=&(c_{Y,Z}\otimes id_X)\circ (id_Y\otimes c_{Z,X}^{-1})
(h_1\otimes h_2\otimes 1)\\
&=&(c_{Y,Z}\otimes id_X)(h_1\otimes h_2S(h_4)\otimes h_3)\\
&=&(h_2S(h_4))_2\otimes [(h_2S(h_4))_3S^{-1}((h_2S(h_4))_1)]\cdot h_1
\otimes h_3, 
\end{eqnarray*}
${\;\;\;\;\;\;}$
$(id_Z \otimes c_{X,Y})\circ (c_{Z,X}^{-1}\otimes id_Y)\circ  
(id_X\otimes c_{Y,Z})(1\otimes h\otimes 1)$
\begin{eqnarray*}
&=&(id_Z \otimes c_{X,Y})\circ (c_{Z,X}^{-1}\otimes id_Y)
(1\otimes 1\otimes h)\\
&=&(id_Z \otimes c_{X,Y})(1\otimes 1\otimes h)\\
&=&1\otimes h_1\otimes h_2, 
\end{eqnarray*}
so we obtain 
\begin{eqnarray*} 
&&(h_2S(h_4))_2\otimes [(h_2S(h_4))_3S^{-1}((h_2S(h_4))_1)]\cdot h_1
\otimes h_3=1\otimes h_1\otimes h_2.
\end{eqnarray*}
By applying $id\otimes \varepsilon \otimes id$ we get 
$\;h_1S(h_3)\otimes h_2=1\otimes h,\;$  
which, by making convolution with $S(h)\otimes 1$, becomes 
$\;S(h_1)h_2S(h_4)\otimes h_3=S(h_1)\otimes h_2,\;$  
and so we obtain $\;S(h_2)\otimes h_1=S(h_1)\otimes h_2,\;$  
which implies $\Delta ^{cop}(h)=\Delta (h)$, i.e. $H$ is cocommutative.
 
We prove now that $H$ is commutative. Note first that  
cocommutativity implies $c_{H_2, H_1}(b\otimes a)=a\otimes b$, for all 
$a, b\in H$. Let now $g, h\in H$; we will apply the pseudosymmetry 
condition (\ref{pseudosym}) for $X=H_1$, $Y=H_2$, $Z=H_2$ on the element 
$1\otimes g\otimes h$:\\[2mm]
${\;\;\;\;\;\;}$
$(c_{Y,Z}\otimes id_X)\circ (id_Y\otimes c_{Z,X}^{-1})\circ 
(c_{X,Y}\otimes id_Z)(1\otimes g\otimes h)$
\begin{eqnarray*}
&=&(c_{Y,Z}\otimes id_X)\circ (id_Y\otimes c_{Z,X}^{-1})
(g_1\otimes g_2\otimes h)\\
&=&(c_{Y,Z}\otimes id_X)(g_1\otimes h\otimes g_2)\\
&=&h_1\otimes h_3g_1S^{-1}(h_2) \otimes g_2, 
\end{eqnarray*}
${\;\;\;\;\;\;}$
$(id_Z \otimes c_{X,Y})\circ (c_{Z,X}^{-1}\otimes id_Y)\circ 
(id_X\otimes c_{Y,Z})(1\otimes g\otimes h)$
\begin{eqnarray*}
&=&(id_Z \otimes c_{X,Y})\circ (c_{Z,X}^{-1}\otimes id_Y)
(1\otimes h_1\otimes h_3gS^{-1}(h_2))\\
&=&(id_Z \otimes c_{X,Y})(h_1\otimes 1\otimes h_3gS^{-1}(h_2))\\
&=&h_1\otimes (h_3gS^{-1}(h_2))_1 \otimes (h_3gS^{-1}(h_2))_2, 
\end{eqnarray*}
and so we obtain 
\begin{eqnarray*}
&&h_1\otimes h_3g_1S^{-1}(h_2)\otimes g_2=
h_1\otimes (h_3gS^{-1}(h_2))_1 \otimes (h_3gS^{-1}(h_2))_2. 
\end{eqnarray*}
By applying $id \otimes \varepsilon \otimes id$ we get 
$\;h_1\otimes h_3gS^{-1}(h_2)=h\otimes g,\;$ 
which implies 
$\;h_3gS^{-1}(h_2)h_1=gh,\;$ 
that is $hg=gh$ and hence $H$ is commutative.  
\end{proof}
\begin{corollary}
For $H$ a commutative and cocommutative Hopf algebra, the double  
braiding 
\begin{eqnarray*}
&&T_{X, Y}(x\otimes y)=(c_{Y, X}\circ c_{X, Y})(x\otimes y)=
y_{(1)}\cdot x_{(0)}\otimes x_{(1)}\cdot y_{(0)}
\end{eqnarray*}
is a strong twine on $_H{\cal YD}^H$.  
\end{corollary}
\begin{definition}
If $H$ is a Hopf algebra and $R\in H\otimes H$ is a quasitriangular 
structure, we will say that $R$ is {\bf pseudotriangular} if 
\begin{eqnarray}
&&R_{12}R^{-1}_{31}R_{23}=R_{23}R^{-1}_{31}R_{12}. \label{pseudoQYB}
\end{eqnarray} 
\end{definition}

If $R$ is a pseudotriangular structure then it is easy to see that the 
braiding on $_H{\cal M}$ given by $c_{M, N}: M\otimes N\rightarrow 
N\otimes M$, $c_{M, N}(m\otimes n)=R^2\cdot n\otimes R^1\cdot m$, is 
pseudosymmetric. Also, it is obvious that if $R$ is triangular (i.e. 
$R_{21}R=1\otimes 1$) then $R$ is pseudotriangular, because in this case 
(\ref{pseudoQYB}) becomes the quantum Yang-Baxter equation 
$R_{12}R_{13}R_{23}=R_{23}R_{13}R_{12}$. We have also the Hopf-algebraic 
counterpart of Theorem \ref{carpseudosym}:
\begin{proposition}\label{carneat}
Let $(H, R)$ be a quasitriangular Hopf algebra. Then $R$ is pseudotriangular 
if and only if the lazy twist $F=R_{21}R$ satisfies the condition 
$F_{12}F_{23}=F_{23}F_{12}$ (i.e. $F$ is {\em neat}, in the 
terminology of \cite{psvo}).
\end{proposition}
\begin{example} {\em 
If $H$ is a commutative Hopf algebra, then any quasitriangular structure 
on $H$ is pseudotriangular. For instance, if $k$ has characteristic zero 
and contains a primitive root of unity of degree $n$, then the group 
algebra of the cyclic group $\mathbb{Z}_n$ admits a certain 
quasitriangular structure  
(constructed in \cite{maj}, \cite{radford}) which is {\em not}  
triangular for $n\geq 3$. Thus, for $n\geq 3$, the category of 
representations of $\mathbb{Z}_n$ admits a pseudosymmetric braiding which 
is not symmetric.}
\end{example}
\begin{remark}{\em  
Let $H$ be a finite dimensional Hopf algebra. It is well-known that the 
category of Yetter-Drinfeld modules $_H{\cal YD}^H$ is braided 
equivalent to  
the category $_{D(H)}{\cal M}$ of left modules over the Drinfeld double of 
$H$ (realized on $H^{*\;cop}\otimes H$ and with quasitriangular  
structure given by   
$R=\sum (\varepsilon \otimes e_i)\otimes (e^i\otimes 1)$, 
where $\{e_i\}$, $\{e^i\}$  
are dual bases in $H$ and $H^*$). Thus, via  
Theorem \ref{co}, we obtain that    
$R$ is pseudotriangular if and only if $H$ is commutative and 
cocommutative. In particular, if $G$ is a finite, noncommutative group 
then ($D(k[G]), R$) is quasitriangular but not pseudotriangular.} 
\end{remark} 
\begin{definition} (\cite{liuzhu}) 
Let $(H, R)$ be a quasitriangular Hopf algebra. The element $R$  
is called {\bf almost-triangular} if $R_{21}R$ is central  
in $H\otimes H$. 
\end{definition}
\begin{remark}
By Proposition \ref{carneat} it follows that an 
almost-triangular structure is pseudotriangular. The converse is not  
true, a counterexample is provided by Proposition \ref{EN} below.
\end{remark}    

Assume now that $char (k)\neq 2$ and  
consider the $2^{n+1}$-dimensional Hopf algebra $E(n)$ generated by 
$c$, $x_1,$$...,$$\;x_n$ with relations $c^2=1$, $x_i^2=0$,  
$x_ic+cx_i=0$ and $x_ix_j+x_jx_i=0$, for all $i, j\in \{1, ..., n\}$, and   
coalgebra structure $\Delta (c)=c\otimes c$,  
$\Delta (x_i)=1\otimes x_i+x_i\otimes c$, for all $i\in \{1, ..., n\}$. 
The quasitriangular structures of $E(n)$ have been classified in 
\cite{panvo}, they are in bijection with $n\times n$ matrices with 
entries in $k$, and moreover the quasitriangular structure $R_A$ 
corresponding to the matrix $A$ is given by an explicit formula, 
generalizing the cases $n=1$ from \cite{radmin} and $n=2$ from 
\cite{gelaki}. By \cite{panvo} and \cite{cc} we know that $R_A$ is 
triangular if and only if the matrix $A$ is symmetric. 
\begin{proposition} \label{EN}
For any $n\times n$ matrix $A$, the quasitriangular structure $R_A$ is 
pseudotriangular, and it is almost-triangular if and only if $A$ is 
symmetric (thus the only almost-triangular structures of $E(n)$ are the 
triangular ones). 
\end{proposition}  
\begin{proof}
We present first an alternative description for the quasitriangular 
structure $R_A$. 
For every $a\in k$ and $i, j\in \{1, ..., n\}$ we define the element
\begin{eqnarray}  
&T_{i,j}(a):=1\otimes 1+a(x_i\otimes cx_j) \in E(n)\otimes E(n). 
\label{Tija}
\end{eqnarray} 
It is easy to see that $T_{i,j}(a)$ is a lazy twist,  
$T_{i,j}(a)T_{i,j}(b)=T_{i,j}(a+b)$ and  
$T_{i,j}(a)T_{k,l}(b)=T_{k,l}(b)T_{i,j}(a)$, for all $a, b\in k$ and 
$i, j, k, l\in \{1, ..., n\}$. If $A=(a_{ij})_{i, j=1, ..., n}$ is an  
$n\times n$ matrix, we define the element 
\begin{eqnarray} 
&&T_A:=\prod _{i,j=1}^nT_{i,j}(a_{ij}) \in E(n)\otimes E(n) \label{TAprod}
\end{eqnarray}
(note that the order of the factors does not matter since they all 
commute). It is clear that if $B$ is another $n\times n$ matrix then 
$T_AT_B=T_{A+B}$. One can also see that the element $T_A$ is given by the 
formula 
\begin{eqnarray}
&&T_A=1\otimes 1+\sum _{\mid P\mid=\mid F\mid}(-1)^
{\frac{\mid P\mid (\mid P\mid-1)}{2}}det (P, F)x_P\otimes c^{\mid P\mid }
x_F, \label{TA}
\end{eqnarray} 
where the sum is made over all nonempty subsets $P$, $F$ of 
$\{1, ..., n\}$ such that $\mid P\mid=\mid F\mid$, and if 
$P=\{i_1<i_2<\dotsb < i_s\}$ and $F=\{j_1<j_2<\dotsb < j_s\}$ then 
$det (P, F)$ is the determinant of the $s\times s$ matrix obtained at 
the intersection of the rows $i_1, ..., i_s$ and columns $j_1, ..., j_s$ of  
the matrix $A$, and $x_P=x_{i_1}\dotsb x_{i_s}$,  
$x_F=x_{j_1}\dotsb x_{j_s}$. In particular we obtain $T_0=1\otimes 1$ and 
$T_A^{-1}=T_{-A}$. 
 
Define now the element 
\begin{eqnarray*}
&&R:=\frac{1}{2}(1\otimes 1+c\otimes 1+1\otimes c-c\otimes c) 
\in E(n)\otimes E(n),
\end{eqnarray*}
which is a triangular structure for $E(n)$.  
From the formula for the quasitriangular structure $R_A$ in \cite{panvo} 
and (\ref{TA}) we immediately obtain 
\begin{eqnarray}  
&&R_A=RT_A. \label{RARTA}
\end{eqnarray}
If we denote by $A^t$ the transpose of a matrix $A$, then we know from 
\cite{cc} that 
\begin{eqnarray}
&&R_A^{-1}=(R_{A^t})_{21}, \label{invRA}
\end{eqnarray}
a consequence of which is the relation  
$(R_A)_{21}R_B=T_{B-A^t}$, for any $n\times n$ matrices $A$ and $B$. We 
record also the obvious relation $R_AT_B=R_{A+B}$, as well as 
$(T_A)_{21}R_B=R_{B-A^t}$. 

Let now $A$ be an $n\times n$ matrix;  
we will prove that $R_A$ is pseudotriangular. In view of (\ref{invRA}), 
what we need to prove is the relation 
\begin{eqnarray}
&&(R_A)_{12}(R_{A^t})_{13}(R_A)_{23}=(R_A)_{23}(R_{A^t})_{13}(R_A)_{12}. 
\label{modif}
\end{eqnarray}
We will actually prove something more general, namely 
\begin{eqnarray}
&&(R_A)_{12}(R_B)_{13}(R_C)_{23}=(R_C)_{23}(R_B)_{13}(R_A)_{12},  
\label{modifgen}
\end{eqnarray}
for any $n\times n$ matrices $A$, $B$ and $C$. We introduce the following  
notation, for $a\in k$ and $i, j\in \{1, ..., n\}$: 
\begin{eqnarray*}
&&T_{i, j}(a)_{12c}:=1\otimes 1\otimes 1+ax_i\otimes cx_j\otimes c, \\
&&T_{i, j}(a)_{1c3}:=1\otimes 1\otimes 1+ax_i\otimes c\otimes cx_j, \\
&&T_{i, j}(a)_{c23}:=1\otimes 1\otimes 1+c\otimes ax_i\otimes cx_j.
\end{eqnarray*} 
By direct computation one can prove the following relations:
\begin{eqnarray*}
&&T_{i,j}(a)_{23}R_{13}=R_{13}T_{i,j}(a)_{c23},\\  
&&T_{i,j}(a)_{c23}R_{12}=R_{12}T_{i,j}(a)_{23}, \\
&&T_{i,j}(a)_{13}R_{12}=R_{12}T_{i,j}(a)_{1c3}, \\
&&T_{i,j}(a)_{13}R_{23}=R_{23}T_{i,j}(a)_{1c3}, \\ 
&&T_{i,j}(a)_{12}R_{13}=R_{13}T_{i,j}(a)_{12c}, \\ 
&&T_{i,j}(a)_{12c}R_{23}=R_{23}T_{i,j}(a)_{12}.
\end{eqnarray*}
One can also see that, for all $i, j, k, l, p, q\in \{1, ..., n\}$ and 
$x, y, z\in k$, all the elements $T_{i, j}(x)_{23}$, $T_{k, l}(y)_{12}$ 
and $T_{p, q}(z)_{1c3}$ commute with each other. Using all these facts 
together with the formulae (\ref{RARTA}) and (\ref{TAprod}) we obtain
\begin{eqnarray*}
&&(R_A)_{12}(R_B)_{13}(R_C)_{23}=R_{12}R_{13}R_{23}(T_A)_{12}(T_B)_{1c3}
(T_C)_{23}, \\
&&(R_C)_{23}(R_B)_{13}(R_A)_{12}=R_{23}R_{13}R_{12}(T_C)_{23}(T_B)_{1c3}
(T_A)_{12},
\end{eqnarray*}
and the right hand sides are equal because of the above-mentioned 
commutation relations together with the fact that $R$ satisfies the 
Yang-Baxter equation. 
 
We prove now that $R_A$ is almost-triangular if and only if  
$A$ is symmetric.  
Let $B$ be an $n\times n$ matrix; it is easy to see that $T_B$ is central  
in $E(n)\otimes E(n)$ if and only if $B=0$, because if $B\neq 0$ then 
$T_B$ does not commute with $1\otimes c$. We have seen above that 
$(R_A)_{21}R_A=T_{A-A^t}$, and so $(R_A)_{21}R_A$ is central if and only 
if $A=A^t$.
\end{proof}
\begin{remark}{\em 
We consider the group ${\mb Z}^2(E(n))$ as in Section \ref{seclaycles}, 
and inside it the set 
$G_n:=\{T_A,\,  R_A\}$, where $A$ is an $n\times n$ matrix. If we denote by 
$*$ the multiplication in ${\mb Z}^2(E(n))$, then we have 
\begin{eqnarray*} 
&&T_A*T_B=T_AT_B=T_{A+B}, \\ 
&&R_A*T_B=R_AT_B=R_{A+B},\\ 
&&T_A*R_B=(T_A)_{21}R_B=R_{B-A^t}, \\
&&R_A*R_B=(R_A)_{21}R_B=T_{B-A^t}, 
\end{eqnarray*}
and so $G_n$ is a subgroup of ${\mb Z}^2(E(n))$ (note that the inverse 
of $R_A$ in this group is $R_{A^t}$).    
The above formulae imply  
$G_n\simeq {\mathbb Z}_2\ltimes (M_n(k), +)$, a semidirect 
product, where the action of ${\mathbb Z}_2$ on $(M_n(k), +)$ is given by  
$A\cdot g=-A^t$ ($g$ is the generator of ${\mathbb Z}_2$), 
and the correspondence is given by $T_A\mapsto (1, A)$, 
$R_A\mapsto (g, A)$.    
For $n=1$ ($E(1)$ is Sweedler's 4-dimensional Hopf  
algebra), one can prove by direct computation that $G_1={\mb Z}^2(E(1))$.} 
\end{remark}
\section{Laycles, pseudotwistors and R-matrices}
\setcounter{equation}{0}
${\;\;\;\;\;}$
We recall the following concept and result from \cite{lpvo}: 
\begin{proposition} \label{tilda} (\cite{lpvo})
Let $\cal C$ be a monoidal category, $A$ an algebra in 
$\cal C$ with multiplication $\mu $ and unit $u$, $T:A\otimes A\rightarrow 
A\otimes A$ a morphism in $\cal C$ such 
that $T\circ (u \otimes id_A)=u\otimes id_A$ and  
$T\circ (id_A\otimes u)=id_A\otimes u$.    
Assume that there exist two morphisms 
$\tilde{T}_1, \tilde{T}_2:A\otimes A\otimes A 
\rightarrow A\otimes A\otimes A$ in $\cal C$ such that   
\begin{eqnarray}
&&(id_A\otimes \mu )\circ \tilde{T}_1\circ (T\otimes id_A)=
T\circ (id_A\otimes \mu ), 
\label{catw1} \\
&&(\mu \otimes id_A)\circ \tilde{T}_2\circ (id_A\otimes T)=
T\circ (\mu \otimes id_A), 
\label{catw2} \\
&&\tilde{T}_1\circ (T\otimes id_A)\circ (id_A\otimes T)= 
\tilde{T}_2\circ (id_A\otimes T)\circ (T\otimes id_A). \label{catw3}
\end{eqnarray}
Then $(A, \mu \circ T, u)$ is also an algebra in ${\cal C}$,   
denoted by $A^T$. The morphism $T$ is called a {\bf pseudotwistor} and 
the two morphisms $\tilde{T}_1$, $\tilde{T}_2$ are called the 
{\bf companions} of $T$. If ${\cal C}$ is the category of $k$-vector 
spaces, $\tilde{T}_1=\tilde{T}_2=T_{13}$ and $T_{12}\circ T_{23}=
T_{23}\circ T_{12}$, then $T$ is called a {\bf twistor} for $A$.  
\end{proposition}
\begin{proposition} \label{layclepseudo}
Let ${\cal C}$ be a monoidal category and $T$ a laycle on ${\cal C}$.  
If $(A, \mu , u)$ is an algebra in ${\cal C}$, then $T_{A, A}$ is a 
pseudotwistor for $A$, with companions $\tilde{T}_1:=T^b_{A, A, A}$ and 
$\tilde{T}_2:=T^f_{A, A, A}$, where $T^b$ and $T^f$ are the 
families defined by (\ref{ela}) and (\ref{fla}).   
\end{proposition}
\begin{proof}
We prove (\ref{catw1}). The naturality of $T$ implies 
$T_{A, A}\circ (id_A\otimes \mu )=(id_A\otimes \mu )\circ T_{A, A\otimes A}$, 
and from (\ref{ela}) we obtain $T_{A, A}\circ (id_A\otimes \mu )=
(id_A\otimes \mu )\circ T^b_{A, A, A}\circ (T_{A, A}\otimes id_A)$, q.e.d. 
Similarly one can prove (\ref{catw2}), while (\ref{catw3}) follows 
immediately by using (\ref{ela}), (\ref{fla}) and (\ref{c1}). 
\end{proof}
\begin{corollary}
If $T$ is a laycle on a monoidal category ${\cal C}$ and 
$(A, \mu , u)$ is an algebra in ${\cal C}$, then 
$(A, \mu \circ T_{A, A}, u)$ is also an algebra in ${\cal C}$. 
\end{corollary}
\begin{remark}
If ${\cal C}$ is a monoidal category and $c$ is a braiding on  
${\cal C}$, then, by \cite{brug}, the double braiding 
$c^2_{X, Y}:=c_{Y, X}\circ c_{X, Y}$ is a twine on ${\cal C}$, in particular 
a laycle. Thus, Proposition \ref{layclepseudo} generalizes the fact 
(proved in \cite{lpvo}, Corollary 6.8) that a double braiding induces a 
pseudotwistor on every algebra in ${\cal C}$. 
\end{remark}
\begin{definition}
Let ${\cal C}$ be a monoidal category, $(A, \mu , u)$ an algebra in 
${\cal C}$ and $T:A\otimes A\rightarrow A\otimes A$ a pseudotwistor with 
companions $\tilde{T}_1$ and $\tilde{T}_2$. We say that $T$ is a 
{\bf strong pseudotwistor} if $T$ is invertible and the following conditions 
are satisfied:
\begin{eqnarray}
&&\tilde{T}_2\circ (id_A\otimes T)=(id_A\otimes T)\circ \tilde{T}_1, 
\label{strongpseudo1} \\
&&\tilde{T}_1\circ (T\otimes id_A)=(T\otimes id_A)\circ \tilde{T}_2. 
\label{strongpseudo2}
\end{eqnarray}
In this case, we denote 
\begin{eqnarray*}
&&T_{A\otimes A, A}:=\tilde{T}_2\circ (id_A\otimes T)=
(id_A\otimes T)\circ \tilde{T}_1, \\
&&T_{A, A\otimes A}:=\tilde{T}_1\circ (T\otimes id_A)=
(T\otimes id_A)\circ \tilde{T}_2. 
\end{eqnarray*}
\end{definition}
\begin{remark}
If $T_{X, Y}$ is a laycle on a monoidal category ${\cal C}$ and 
$(A, \mu , u)$ is an algebra in ${\cal C}$, then, by (\ref{c1}), 
it follows that $T_{A, A}$ is a strong pseudotwistor for $A$.
\end{remark}
\begin{lemma}
If $T$ is a strong pseudotwistor, then the  
following relations hold:
\begin{eqnarray}
&&(T\otimes id_A)\circ T_{A\otimes A, A}=T_{A\otimes A, A}\circ 
(T\otimes id_A), \label{T1} \\
&&(id_A\otimes T)\circ T_{A, A\otimes A}=T_{A, A\otimes A}\circ 
(id_A\otimes T), \label{T2} \\
&&T_{A\otimes A, A}\circ (T\otimes id_A)=T_{A, A\otimes A}\circ (id_A
\otimes T), \label{T3}\\
&&(T\otimes id_A)\circ \tilde{T}_2\circ (id_A\otimes T)=
(id_A\otimes T)\circ \tilde{T}_1\circ (T\otimes id_A). \label{T4}
\end{eqnarray}
\end{lemma}
\begin{proof}
Straightforward computation, using (\ref{strongpseudo1}), 
(\ref{strongpseudo2}) and (\ref{catw3}). 
\end{proof}

Our next results are the analogues for pseudotwistors of the facts that 
composition of laycles is a laycle and the inverse of a laycle is a laycle. 
\begin{proposition}
Let ${\cal C}$ be a monoidal category, $(A, \mu , u)$ an algebra in 
${\cal C}$ and $T, D:A\otimes A\rightarrow A\otimes A$ two strong 
pseudotwistors for $A$, such that 
\begin{eqnarray}
&&D_{A, A\otimes A}\circ (id_A\otimes T)=(id_A\otimes T)\circ 
D_{A, A\otimes A}, \label{partic1} \\
&&D_{A\otimes A, A}\circ (T\otimes id_A)=(T\otimes id_A)\circ 
D_{A\otimes A, A}. \label{partic2} 
\end{eqnarray} 
Then $U:=T\circ D$ is a pseudotwistor for $A$, with companions $\tilde{U}_1:= 
T_{A, A\otimes A}\circ \tilde{D}_1\circ (T^{-1}\otimes id_A)$ and 
$\tilde{U}_2:=T_{A\otimes A, A}\circ \tilde{D}_2\circ (id_A\otimes T^{-1})$. 
If moreover we have 
\begin{eqnarray}
&&T_{A, A\otimes A}\circ (id_A\otimes D)=(id_A\otimes D)\circ 
T_{A, A\otimes A}, \label{partic3} \\
&&T_{A\otimes A, A}\circ (D\otimes id_A)=(D\otimes id_A)\circ 
T_{A\otimes A, A},  \label{partic4} 
\end{eqnarray} 
then $U$ is also a strong pseudotwistor.
\end{proposition}
\begin{proof}
We check (\ref{catw1})--(\ref{catw3}) for $U$:
\begin{eqnarray*}
U\circ (id_A\otimes \mu )&=&T\circ D\circ (id_A\otimes \mu )\\
&\overset{(\ref{catw1})}{=}&(id_A\otimes \mu )\circ \tilde{T}_1\circ 
(T\otimes id_A)\circ \tilde{D}_1\circ (D\otimes id_A)\\
&=&(id_A\otimes \mu )\circ T_{A, A\otimes A}\circ \tilde{D}_1\circ 
(D\otimes id_A)\\
&=&(id_A\otimes \mu )\circ \tilde{U}_1\circ (U\otimes id_A), 
\end{eqnarray*}
\begin{eqnarray*}
U\circ (\mu \otimes id_A)&=&T\circ D\circ (\mu \otimes id_A)\\
&\overset{(\ref{catw2})}{=}&(\mu \otimes id_A)\circ \tilde{T}_2\circ 
(id_A\otimes T)\circ \tilde{D}_2\circ (id_A\otimes D)\\
&=&(\mu \otimes id_A)\circ T_{A\otimes A, A}\circ \tilde{D}_2\circ 
(id_A\otimes D)\\
&=&(\mu \otimes id_A)\circ \tilde{U}_2\circ (id_A\otimes U),
\end{eqnarray*}
\begin{eqnarray*}
\tilde{U}_1\circ (U\otimes id_A)\circ (id_A\otimes U)&=&
T_{A, A\otimes A}\circ \tilde{D}_1\circ (D\otimes id_A)
\circ (id_A\otimes T)\circ (id_A\otimes D)\\
&=&T_{A, A\otimes A}\circ D_{A, A\otimes A}\circ (id_A\otimes T)\circ 
(id_A\otimes D)\\
&\overset{(\ref{partic1})}{=}&T_{A, A\otimes A}\circ (id_A\otimes T)\circ 
D_{A, A\otimes A}\circ (id_A\otimes D)\\
&\overset{(\ref{T3})}{=}&T_{A\otimes A, A}\circ (T\otimes id_A)\circ 
D_{A\otimes A, A}\circ (D\otimes id_A)\\
&\overset{(\ref{partic2})}{=}&T_{A\otimes A, A}\circ D_{A\otimes A, A}\circ 
(T\otimes id_A)\circ (D\otimes id_A)\\
&=&T_{A\otimes A, A}\circ \tilde{D}_2\circ (id_A\otimes D)
\circ (T\otimes id_A)\circ (D\otimes id_A)\\
&=&\tilde{U}_2\circ (id_A\otimes U)\circ (U\otimes id_A),  
\end{eqnarray*}
proving that $U$ is a pseudotwistor for $A$. We assume now that  
(\ref{partic3}) and (\ref{partic4}) hold and we prove (\ref{strongpseudo1}) 
and (\ref{strongpseudo2}) for $U$:
\begin{eqnarray*}
(id_A\otimes U)\circ \tilde{U}_1&=&(id_A\otimes T)\circ (id_A\otimes D)
\circ T_{A, A\otimes A}\circ \tilde{D}_1\circ (T^{-1}\otimes id_A)\\
&\overset{(\ref{partic3})}{=}&(id_A\otimes T)\circ T_{A, A\otimes A}\circ 
(id_A\otimes D)\circ \tilde{D}_1\circ (T^{-1}\otimes id_A)\\
&=&(id_A\otimes T)\circ T_{A, A\otimes A}\circ 
D_{A\otimes A, A}\circ (T^{-1}\otimes id_A)\\ 
&\overset{(\ref{partic2})}{=}&(id_A\otimes T)\circ T_{A, A\otimes A}\circ 
(T^{-1}\otimes id_A)\circ D_{A\otimes A, A}\\ 
&\overset{(\ref{T3}),\; (\ref{T2})}{=}&T_{A\otimes A, A}\circ 
D_{A\otimes A, A}\\
&=&T_{A\otimes A, A}\circ \tilde{D}_2\circ (id_A\otimes D)\\
&=&\tilde{U}_2\circ (id_A\otimes U), 
\end{eqnarray*}
\begin{eqnarray*}
(U\otimes id_A)\circ \tilde{U}_2&=&(T\otimes id_A)\circ (D\otimes id_A)
\circ T_{A\otimes A, A}\circ \tilde{D}_2\circ (id_A\otimes T^{-1})\\
&\overset{(\ref{partic4})}{=}&
(T\otimes id_A)\circ T_{A\otimes A, A}\circ (D\otimes id_A)
\circ \tilde{D}_2\circ (id_A\otimes T^{-1})\\
&=&(T\otimes id_A)\circ T_{A\otimes A, A}\circ 
D_{A, A\otimes A}\circ (id_A\otimes T^{-1})\\
&\overset{(\ref{partic1})}{=}&
(T\otimes id_A)\circ T_{A\otimes A, A}\circ 
(id_A\otimes T^{-1})\circ D_{A, A\otimes A}\\
&\overset{(\ref{T3}),\; (\ref{T1})}{=}&
T_{A, A\otimes A}\circ D_{A, A\otimes A}\\
&=&T_{A, A\otimes A}\circ \tilde{D}_1\circ (D\otimes id_A)\\
&=&\tilde{U}_1\circ (U\otimes id_A),
\end{eqnarray*}
showing that $U$ is a strong pseudotwistor.
\end{proof}
\begin{corollary}
If $T$ is a strong pseudotwistor for an algebra $(A, \mu , u)$ in 
a monoidal category ${\cal C}$, then $T\circ T$ is also a strong 
pseudotwistor for $A$.
\end{corollary}
\begin{proposition}
Let ${\cal C}$ be a monoidal category, $(A, \mu , u)$ an algebra in 
${\cal C}$ and $T:A\otimes A\rightarrow A\otimes A$ a strong 
pseudotwistor for $A$ such that the companions $\tilde{T}_1$ and 
$\tilde{T}_2$ are invertible. Then the inverse $V:=T^{-1}$ is also 
a strong pseudotwistor for $A$, with companions 
$\tilde{V}_1=\tilde{T}_2^{-1}$ and $\tilde{V}_2=\tilde{T}_1^{-1}$.
\end{proposition}
\begin{proof}
Straightforward computation, using (\ref{catw1})--(\ref{catw3}) for $T$ 
together with (\ref{strongpseudo1}) and (\ref{strongpseudo2}).
\end{proof}
\begin{remark}\label{lulu}
Let ${\cal C}$ be a monoidal category, $(A, \mu , u)$ an algebra in 
${\cal C}$, $T:A\otimes A\rightarrow A\otimes A$ a strong pseudotwistor 
for $A$ and ${\cal D}$ a laycle on ${\cal C}$. Then $T$ and 
$D:={\cal D}_{A, A}$ satisfy (\ref{partic1}) and (\ref{partic2}), 
hence $T\circ D$ is a pseudotwistor for $A$. 
\end{remark}

Our next result is the analogue for pseudotwistors of the fact  
that if $\sigma $, $\sigma '$ are cohomologous lazy cocycles on a 
Hopf algebra $H$ then the algebras $H(\sigma )$ and $H(\sigma ')$ are 
isomorphic:
\begin{proposition}
Let ${\cal C}$ be a monoidal category, $(A, \mu , u)$ an algebra in 
${\cal C}$, $T:A\otimes A\rightarrow A\otimes A$ a strong pseudotwistor 
for $A$ and $R_X:X\rightarrow X$ a family of natural isomorphisms in 
${\cal C}$ such that $R_I=id_I$. Then we have an algebra isomorphism 
\begin{eqnarray*}
&&R_A:A^{T\circ D^1(R)_{A, A}}\simeq A^T.
\end{eqnarray*} 
\end{proposition}
\begin{proof}
Note first that $T\circ D^1(R)_{A, A}$ is a pseudotwistor by  
Remark \ref{lulu}. We compute: 
\begin{eqnarray*}
R_A\circ \mu \circ T\circ D^1(R)_{A, A}&=&R_A\circ \mu \circ T\circ 
R_{A\otimes A}^{-1}\circ (R_A\otimes R_A)\\
&\overset{naturality\;of\;R}{=}&
\mu \circ R_{A\otimes A}\circ T\circ R_{A\otimes A}^{-1}\circ 
(R_A\otimes R_A)\\
&\overset{naturality\;of\;R}{=}&\mu \circ T\circ (R_A\otimes R_A), 
\end{eqnarray*}
finishing the proof.
\end{proof}

If $T$ is a laycle on a monoidal category ${\cal C}$, then, by \cite{brug}, 
$T$ is a twine if and only if the following condition is satisfied:
\begin{eqnarray}
&&(T^f_{X, Y, Z}\otimes id_W)\circ (id_X\otimes T^b_{Y, Z, W})=
(id_X\otimes T^b_{Y, Z, W})\circ (T^f_{X, Y, Z}\otimes id_W), 
\label{caract}
\end{eqnarray}
for all $X, Y, Z, W\in {\cal C}$. Note that the families $T^f$, $T^b$ 
coincide respectively to the families $A$, $B$ from the Definition 
\ref{purebraided} of a pure-braided structure, and (\ref{caract}) 
coincides with (\ref{cab}). We are thus led to the following 
concept and terminology:
\begin{definition}
Let ${\cal C}$ be a monoidal category, $A$ an algebra in ${\cal C}$ and 
$T:A\otimes A\rightarrow A\otimes A$ a pseudotwistor with companions 
$\tilde{T}_1$ and $\tilde{T}_2$. We call $T$ a {\bf pure pseudotwistor} 
if 
\begin{eqnarray}
&&(\tilde{T}_2\otimes id_A)\circ (id_A\otimes \tilde{T}_1)=
(id_A\otimes \tilde{T}_1)\circ (\tilde{T}_2\otimes id_A). \label{purepseudo}
\end{eqnarray}
\end{definition}
\begin{corollary}
A twine on a monoidal category ${\cal C}$ induces a pure pseudotwistor 
on every algebra $A$ in ${\cal C}$. In particular, if $c$ is a braiding 
on ${\cal C}$ then $c^2_{A, A}$ is a pure pseudotwistor for $A$.  
\end{corollary}
\begin{remark}{\em 
Obviously, a pseudotwistor for which $\tilde{T}_1=\tilde{T}_2=
id_{A\otimes A\otimes A}$ is pure. Here are some  
concrete (but nonunital) examples of such pseudotwistors:\\
(i) take $A$ an associative algebra,  
$R=R^1\otimes R^2\in A\otimes A$ and define $T(a\otimes b)=aR^1\otimes 
R^2b$, for all $a, b\in A$.\\
(ii) take $A$ an associative algebra, $f:A\rightarrow A$  
a linear map satisfying $f(ab)=af(b)$ for all $a, b\in A$, and  
$T(a\otimes b)=f(a)\otimes b$. If instead $f$ satisfies $f(ab)=f(a)b$, then 
take $T(a\otimes b)=a\otimes f(b)$.\\
(iii) take $A$ an associative algebra, $\delta :A\rightarrow A\otimes A$ 
a linear map such that $\delta (ab)=(a\otimes 1)\delta (b)$ for all 
$a, b\in A$ and $T:A\otimes A\rightarrow A\otimes A$, $T(a\otimes b)=
\delta (a)(1\otimes b)$. If instead $\delta $ satisfies 
$\delta (ab)=\delta (a)(1\otimes b)$, then take $T(a\otimes b)=
(a\otimes 1)\delta (b)$. 

Note that example (i) was inspired by a construction in 
\cite{street},  
while (ii) and (iii) are related to some constructions in \cite{leroux} 
involving so-called (anti-) dipterous algebras.}
\end{remark} 
\begin{example} {\em 
If $A$ is an associative algebra and $T:A\otimes A\rightarrow A\otimes A$ is 
a twistor, then it is easy to see that $T$ is pure. }
\end{example}
\begin{example} {\em 
We recall some facts from \cite{lpvo}. 
Let $(\Omega , d)$ be a DG algebra, that is $\Omega =\bigoplus _{n\geq 0}
\Omega ^n$ is a graded algebra and $d:\Omega \rightarrow \Omega $ is a 
linear map with $d(\Omega ^n)\subseteq \Omega ^{n+1}$ for all  
$n\geq 0$, $d^2=0$ and $d(\omega \zeta )=d(\omega )\zeta +(-1)^{|\omega |}
\omega d(\zeta )$ for all homogeneous $\omega $ and $\zeta $, where 
$|\omega |$ is the degree of $\omega $. The Fedosov product  
(\cite{fedosov}, \cite{cuqui}), given by 
$\;\omega \circ \zeta =\omega \zeta -(-1)^{|\omega |}
d(\omega )d(\zeta )\;$,    
for homogeneous $\omega $ and $\zeta $, gives a new associative algebra 
structure on $\Omega $. We consider ${\cal C}$ to be the monoidal category 
of $\mathbb{Z}_2$-graded vector spaces, and regard $\Omega $ as a  
$\mathbb{Z}_2$-graded algebra (i.e. an algebra in ${\cal C}$) by putting 
even components in degree zero and odd components in degree one. 
Define the linear map  
\begin{eqnarray*}
&&T:\Omega \otimes \Omega \rightarrow \Omega \otimes \Omega , \;\;\;
T(\omega \otimes \zeta )=\omega \otimes \zeta -(-1)^{|\omega |}
d(\omega )\otimes d(\zeta ), 
\end{eqnarray*}
for homogeneous $\omega $ and $\zeta $. Then $T$ is a pseudotwistor 
for $\Omega $ in ${\cal C}$, affording the Fedosov product. Its 
companions are given (for homogeneous $\omega $, $\zeta $, $\eta $) by 
\begin{eqnarray*}
&&\tilde{T}_1(\omega \otimes \zeta \otimes \eta )=
\tilde{T}_2(\omega \otimes \zeta \otimes \eta )
=\omega \otimes \zeta  
\otimes \eta -(-1)^{|\omega |+|\zeta |}d(\omega )\otimes \zeta \otimes 
d(\eta ).
\end{eqnarray*} 
We claim that $T$ is a pure pseudotwistor. Indeed, a straightforward 
computation shows that 
\begin{eqnarray*}
&&(\tilde{T}_2\otimes id)\circ (id\otimes \tilde{T}_1)
(\omega \otimes \zeta \otimes \eta \otimes \nu )=
(id \otimes \tilde{T}_1)\circ (\tilde{T}_2\otimes id)
(\omega \otimes \zeta \otimes \eta \otimes \nu )\\
&&\;\;\;\;\;\;\;=\omega \otimes \zeta \otimes \eta \otimes \nu  
-(-1)^{|\omega |+|\zeta |}d(\omega )\otimes \zeta \otimes d(\eta )
\otimes \nu \\
&&\;\;\;\;\;\;\;-(-1)^{|\zeta |+|\eta |}\omega \otimes d(\zeta )\otimes \eta  
\otimes d(\nu )
-(-1)^{|\omega |+|\eta |}d(\omega )\otimes d(\zeta )\otimes d(\eta )
\otimes d(\nu ), 
\end{eqnarray*} 
for all homogeneous $\omega $, $\zeta $, $\eta $, $\nu $.}
\end{example}

We recall the following result from \cite{lpvo}:
\begin{proposition}(\cite{lpvo}) \label{twm}
Let $(A, \mu , u)$ be an algebra in a monoidal category  
${\cal C}$, let $R, P:A\otimes A\rightarrow A\otimes A$ twisting maps  
between $A$ and itself such that $R$ is invertible, and assume that 
\begin{eqnarray}
&&(P\otimes id_A)\circ (id_A\otimes P)\circ (P\otimes id_A)=
(id_A\otimes P)\circ (P\otimes id_A)\circ (id_A\otimes P), \label{br1} \\
&&(R\otimes id_A)\circ (id_A\otimes R)\circ (R\otimes id_A)=
(id_A\otimes R)\circ (R\otimes id_A)\circ (id_A\otimes R), \label{br2} \\
&&(P\otimes id_A)\circ (id_A\otimes P)\circ (R\otimes id_A)=
(id_A\otimes R)\circ (P\otimes id_A)\circ (id_A\otimes P), \label{br3} \\
&&(R\otimes id_A)\circ (id_A\otimes P)\circ (P\otimes id_A)=
(id_A\otimes P)\circ (P\otimes id_A)\circ (id_A\otimes R). \label{br4} 
\end{eqnarray}
Define $T:A\otimes A\rightarrow A\otimes A$, $T:=R^{-1}\circ P$. Then 
$T$ is a pseudotwistor with companions 
\begin{eqnarray*}
&&\tilde{T}_1=(R^{-1}\otimes id_A)\circ (id_A\otimes T)\circ 
(R\otimes id_A), \;\;\;
\tilde{T}_2=(id_A\otimes R^{-1})\circ (T\otimes id_A)\circ (id_A\otimes R). 
\end{eqnarray*}
\end{proposition}

Our next result is the analogue for pseudotwistors of the fact from 
\cite{brug} that a family of the type $T_{X, Y}=c'_{Y, X}\circ c_{X, Y}$, 
with $c$, $c'$ braidings, is a twine: 
\begin{proposition} \label{purestrong}
Assume that the hypotheses of Proposition \ref{twm} hold. Then:\\
(i) $T$ is a pure pseudotwistor; \\
(ii) assume that moreover $P$ is also invertible and 
\begin{eqnarray} 
&&(P\otimes id_A)\circ (id_A\otimes R)\circ (R\otimes id_A)=
(id_A\otimes R)\circ (R\otimes id_A)\circ (id_A\otimes P), \label{br5} \\
&&(R\otimes id_A)\circ (id_A\otimes R)\circ (P\otimes id_A)=
(id_A\otimes P)\circ (R\otimes id_A)\circ (id_A\otimes R) \label{br6} 
\end{eqnarray}
(these conditions appear in \cite{lpvo} too and they imply that $R$ is 
also a twisting map between $A^T$ and itself). Then $T$ is a strong 
pseudotwistor. 
\end{proposition}
\begin{proof}
We check (\ref{purepseudo}): 
\begin{eqnarray*}
(\tilde{T}_2\otimes id_A)\circ (id_A\otimes \tilde{T}_1)&=&
(id_A\otimes R^{-1}\otimes id_A)\circ (T\otimes id_A\otimes id_A)\circ 
(id_A\otimes R\otimes id_A)\\
&&\circ (id_A\otimes R^{-1}\otimes id_A)\circ (id_A\otimes id_A\otimes T)
\circ (id_A\otimes R\otimes id_A)\\
&=&(id_A\otimes R^{-1}\otimes id_A)\circ (id_A\otimes id_A\otimes T)
\circ (T\otimes id_A\otimes id_A)\\
&&\circ (id_A\otimes R\otimes id_A)\\
&=&(id_A\otimes R^{-1}\otimes id_A)\circ (id_A\otimes id_A\otimes T)\circ 
(id_A\otimes R\otimes id_A)\\
&&\circ (id_A\otimes R^{-1}\otimes id_A)
\circ (T\otimes id_A\otimes id_A)\circ (id_A\otimes R\otimes id_A)\\
&=&(id_A\otimes \tilde{T}_1)\circ (\tilde{T}_2\otimes id_A).
\end{eqnarray*}
Assume now that $P$ is invertible and (\ref{br5}), (\ref{br6}) hold.  
Obviously $T$ is invertible, and we only have to check (\ref{strongpseudo1}) 
and (\ref{strongpseudo2}):
\begin{eqnarray*}
\tilde{T}_2\circ (id_A\otimes T)&=&
(id_A\otimes R^{-1})\circ (T\otimes id_A)\circ (id_A\otimes R)\circ 
(id_A\otimes T)\\
&=&(id_A\otimes R^{-1})\circ (R^{-1}\otimes id_A)\circ (P\otimes id_A)
\circ (id_A\otimes P)\\
&\overset{(\ref{br5})}{=}&
(id_A\otimes R^{-1})\circ (id_A\otimes P)\circ (R^{-1}\otimes id_A)\circ 
(id_A\otimes R^{-1})\\
&&(P^{-1}\otimes id_A)\circ (id_A\otimes R)
\circ (P\otimes id_A)\circ (id_A\otimes P)\\
&\overset{(\ref{br3})}{=}&
(id_A\otimes R^{-1})\circ (id_A\otimes P)\circ (R^{-1}\otimes id_A)\\
&&\circ (id_A\otimes R^{-1})\circ (id_A\otimes P)\circ (R\otimes id_A)\\
&=&(id_A\otimes T)\circ \tilde{T}_1, 
\end{eqnarray*}
\begin{eqnarray*}
\tilde{T}_1\circ (T\otimes id_A)&=&(R^{-1}\otimes id_A)\circ (id_A\otimes T)
\circ (R\otimes id_A)\circ (T\otimes id_A)\\
&=&(R^{-1}\otimes id_A)\circ (id_A\otimes R^{-1})\circ (id_A\otimes P)
\circ (P\otimes id_A)\\
&\overset{(\ref{br6})}{=}&
(R^{-1}\otimes id_A)\circ (P\otimes id_A)\circ (id_A\otimes R^{-1})\circ 
(R^{-1}\otimes id_A)\\
&&(id_A\otimes P^{-1})\circ (R\otimes id_A)
\circ (id_A\otimes P)\circ (P\otimes id_A)\\
&\overset{(\ref{br4})}{=}&
(R^{-1}\otimes id_A)\circ (P\otimes id_A)\circ (id_A\otimes R^{-1})\\
&&\circ (R^{-1}\otimes id_A)\circ (P\otimes id_A)\circ (id_A\otimes R)\\
&=&(T\otimes id_A)\circ \tilde{T}_2,
\end{eqnarray*}
finishing the proof.
\end{proof}
\begin{remark}
If $T$ is a {\em braided twistor} as introduced in \cite{lpvo}, 
a computation identical to the one in the proof of Proposition 
\ref{purestrong} (i) shows that $T$ is a pure pseudotwistor.   
\end{remark}
\begin{example} {\em 
Let $A$ be an algebra and ${\cal F}$ a braid system over $A$ as introduced 
by Durdevich in \cite{durdevich}, that is a collection of bijective  
twisting maps between $A$ and itself, satisfying the condition 
\begin{eqnarray*}
&&(\alpha \otimes id_A)\circ (id_A\otimes \beta )\circ (\gamma \otimes id_A)= 
(id_A\otimes \gamma )\circ (\beta \otimes id_A)\circ (id_A\otimes \alpha ), 
\;\;\;\forall \;\alpha , \beta , \gamma \in {\cal F}.
\end{eqnarray*}
For $\alpha , \beta \in {\cal F}$ define the map  
$T_{\alpha , \beta }:A\otimes A\rightarrow A\otimes A$,  
$T_{\alpha , \beta }:=\alpha ^{-1}\circ \beta $. By \cite{lpvo} we know  
that $T$ is a pseudotwistor for $A$, and by Proposition 
\ref{purestrong} it follows that it is a pure strong pseudotwistor.} 
\end{example}

We introduce now the categorical version of Borcherds' $R$-matrices:
\begin{proposition} 
Let $\cal C$ be a monoidal category, $(A, \mu , u)$ an algebra in 
$\cal C$ and $T:A\otimes A\rightarrow  
A\otimes A$ a morphism in $\cal C$ such 
that $T\circ (u \otimes id_A)=u\otimes id_A$ and   
$T\circ (id_A\otimes u)=id_A\otimes u$.    
Assume that there exist two morphisms 
$\overline{T}_1, \overline{T}_2:A\otimes A\otimes A  
\rightarrow A\otimes A\otimes A$ in $\cal C$ such that   
\begin{eqnarray}
&&(id_A\otimes \mu )\circ (T\otimes id_A)\circ \overline{T}_1=
T\circ (id_A\otimes \mu ), 
\label{Rmatrix1} \\
&&(\mu \otimes id_A)\circ (id_A\otimes T)\circ \overline{T}_2=
T\circ (\mu \otimes id_A), 
\label{Rmatrix2} \\
&&(T\otimes id_A)\circ \overline{T}_1\circ (id_A\otimes T)= 
(id_A\otimes T)\circ \overline{T}_2\circ (T\otimes id_A). \label{Rmatrix3}
\end{eqnarray}
Then $(A, \mu \circ T, u)$ is also an algebra in ${\cal C}$,   
denoted by $A^T$. The morphism $T$ is called an {\bf $R$-matrix} and  
the two morphisms $\overline{T}_1$, $\overline{T}_2$ are called the 
{\bf companions} of $T$. Obviously, the original concept of $R$-matrix 
is obtained for ${\cal C}$ being the category of $k$-vector spaces and 
$\overline{T}_1=\overline{T}_2=T_{13}$. 
\end{proposition}
\begin{proof}
Obviously $u$ is a unit for $(A, \mu \circ T)$;   
we check the associativity of $\mu \circ T$:
\begin{eqnarray*}
(\mu \circ T)\circ ((\mu \circ T)\otimes id_A)&=&(\mu \circ T)\circ  
(\mu \otimes id_A)\circ (T\otimes id_A)\\
&\overset{(\ref{Rmatrix2})}{=}&\mu \circ (\mu \otimes id_A)\circ   
(id_A\otimes T)\circ \overline{T}_2\circ (T\otimes id_A)\\
&\overset{(\ref{Rmatrix3})}{=}&\mu \circ (\mu \otimes id_A)\circ   
(T\otimes id_A)\circ \overline{T}_1\circ (id_A\otimes T)\\
&=&\mu \circ (id_A\otimes \mu )\circ (T\otimes id_A)\circ \overline{T}_1 
\circ (id_A\otimes T)\\
&\overset{(\ref{Rmatrix1})}{=}& \mu \circ T\circ (id_A\otimes \mu )
\circ (id_A\otimes T)\\
&=&(\mu \circ T)\circ (id_A\otimes (\mu \circ T)),
\end{eqnarray*}
finishing the proof.
\end{proof}
\begin{proposition}
Let ${\cal C}$ be a monoidal category, $(A, \mu , u)$ an algebra in 
${\cal C}$ and $T:A\otimes A\rightarrow A\otimes A$ an {\em invertible}  
morphism in ${\cal C}$. Then $T$ is a pseudotwistor if and only if it 
is an $R$-matrix. More precisely, if $T$ is a pseudotwistor with 
companions $\tilde{T}_1$, $\tilde{T}_2$ then $T$ is an $R$-matrix with 
companions $\overline{T}_1=(T^{-1}\otimes id_A)\circ \tilde{T}_1\circ 
(T\otimes id_A)$ and $\overline{T}_2=(id_A\otimes T^{-1})\circ \tilde{T}_2
\circ (id_A\otimes T)$; conversely, if $T$ is an $R$-matrix with 
companions $\overline{T}_1$, $\overline{T}_2$ then $T$ is a pseudotwistor 
with companions $\tilde{T}_1=(T\otimes id_A)\circ \overline{T}_1\circ 
(T^{-1}\otimes id_A)$ and $\tilde{T}_2=(id_A\otimes T)\circ \overline{T}_2
\circ (id_A\otimes T^{-1})$.   
\end{proposition}
\begin{proof}
Straightforward computation.
\end{proof}
\begin{corollary}
Let ${\cal C}$ be a monoidal category and $T$ a laycle on ${\cal C}$. 
If $(A, \mu , u)$ is an algebra in ${\cal C}$, then $T_{A, A}$ is an 
$R$-matrix for $A$, with companions $\overline{T}_1:=T^f_{A, A, A}$ 
and $\overline{T}_2:=T^b_{A, A, A}$, where $T^b$ and $T^f$ are the 
families defined by (\ref{ela}) and (\ref{fla}). 
\end{corollary}
\section{A characterization of generalized double braidings}
\setcounter{equation}{0}
${\;\;\;\;\;}$
Let ${\cal C}$ be a monoidal category and $A$ an algebra in ${\cal C}$. 
If $T$ is a pseudotwistor for $A$ and $R:A\otimes A\rightarrow A\otimes A$ 
is an invertible twisting map such that the companions of $T$ are given 
by the formulae 
\begin{eqnarray}
&&\tilde{T}_1=(R^{-1}\otimes id_A)\circ (id_A\otimes T)\circ 
(R\otimes id_A),  \label{ttilda1} \\
&&\tilde{T}_2=(id_A\otimes R^{-1})\circ (T\otimes id_A)\circ 
(id_A\otimes R), \label{ttilda2} 
\end{eqnarray}
then, by \cite{lpvo}, Theorem 6.6, it follows that $R\circ T$ is a 
twisting map between $A$ and itself. This result has the following 
categorical analogue, with laycles replacing pseudotwistors and braidings 
replacing twisting maps: 
\begin{theorem} \label{analog}
Let ${\cal C}$ be a monoidal category, $T$ a laycle and $d$ a 
braiding on ${\cal C}$, such that for all $X, Y, Z\in {\cal C}$ the 
following relations hold:
\begin{eqnarray}
&&T_{X\otimes Y, Z}=(id_X\otimes T_{Y, Z})\circ (d_{X, Y}^{-1}\otimes id_Z)
\circ (id_Y\otimes T_{X, Z})\circ (d_{X, Y}\otimes id_Z), \label{comp1} \\
&&T_{X, Y\otimes Z}=(T_{X, Y}\otimes id_Z)\circ (id_X\otimes d_{Y, Z}^{-1})
\circ (T_{X, Z}\otimes id_Y)\circ (id_X\otimes d_{Y, Z}). \label{comp2} 
\end{eqnarray}
Then the families $d'_{X, Y}:= d_{X, Y}\circ T_{X, Y}$ and 
$d''_{X, Y}:=T_{Y, X}\circ d_{X, Y}$ are also braidings on ${\cal C}$. 
Moreover, $T$ is a twine and $d''_{X, Y}=T_{Y, X}\circ d'_{X, Y}
\circ T_{X, Y}^{-1}$ (thus $({\cal C}, d')$ and $({\cal C}, d'')$ are 
braided isomorphic).    
\end{theorem}
\begin{proof}
Note first that (\ref{comp1}) and (\ref{comp2}) are the analogues of 
(\ref{ttilda1}) and (\ref{ttilda2}), because they are respectively 
equivalent to 
\begin{eqnarray*}
&&T^b_{X, Y, Z}=(d_{X, Y}^{-1}\otimes id_Z)\circ (id_Y\otimes T_{X, Z})\circ 
(d_{X, Y}\otimes id_Z), \\
&&T^f_{X, Y, Z}=(id_X\otimes d_{Y, Z}^{-1})\circ (T_{X, Z}\otimes id_Y)\circ 
(id_X\otimes d_{Y, Z}). 
\end{eqnarray*}
Also, as consequences of (\ref{c1}), (\ref{comp1}) and (\ref{comp2}) we 
obtain the following relations:
\begin{eqnarray}
&&T_{X, Y\otimes Z}
=(d_{X, Y}^{-1}\otimes id_Z)\circ (id_Y\otimes T_{X, Z})\circ 
(d_{X, Y}\otimes id_Z)\circ (T_{X, Y}\otimes id_Z), \label{conseqcomp1}\\ 
&&T_{X\otimes Y, Z}
=(id_X\otimes d_{Y, Z}^{-1})\circ (T_{X, Z}\otimes id_Y)\circ 
(id_X\otimes d_{Y, Z})\circ (id_X\otimes T_{Y, Z}). 
\label{conseqcomp2}
\end{eqnarray}
Now we check (\ref{braid1}) and (\ref{braid2}) for $d'$:
\begin{eqnarray*}
d'_{X, Y\otimes Z}&=&d_{X, Y\otimes Z}\circ T_{X, Y\otimes Z}\\
&\overset{(\ref{braid1}),\; (\ref{comp2})}{=}&
(id_Y\otimes d_{X, Z})\circ (d_{X, Y}\otimes id_Z)\circ 
(T_{X, Y}\otimes id_Z)\circ (id_X\otimes d_{Y, Z}^{-1})\\
&&\circ (T_{X, Z}\otimes id_Y)\circ (id_X\otimes d_{Y, Z})\\
&\overset{(\ref{conseqcomp1})}{=}&
(id_Y\otimes d_{X, Z})\circ (id_Y\otimes T_{X, Z})\circ 
(d_{X, Y}\otimes id_Z)\circ (T_{X, Y}\otimes id_Z)\\
&=&(id_Y\otimes d'_{X, Z})\circ (d'_{X, Y}\otimes id_Z), 
\end{eqnarray*}
\begin{eqnarray*}
d'_{X\otimes Y, Z}&=&d_{X\otimes Y, Z}\circ T_{X\otimes Y, Z}\\
&\overset{(\ref{braid2}),\; (\ref{comp1})}{=}&
(d_{X, Z}\otimes id_Y)\circ (id_X\otimes d_{Y, Z})\circ 
(id_X\otimes T_{Y, Z})\circ (d_{X, Y}^{-1}\otimes id_Z)\\
&&\circ (id_Y\otimes T_{X, Z})\circ (d_{X, Y}\otimes id_Z)\\
&\overset{(\ref{conseqcomp2})}{=}&
(d_{X, Z}\otimes id_Y)\circ (T_{X, Z}\otimes id_Y)\circ 
(id_X\otimes d_{Y, Z})\circ (id_X\otimes T_{Y, Z})\\
&=&(d'_{X, Z}\otimes id_Y)\circ (id_X\otimes d'_{Y, Z}). 
\end{eqnarray*}
Thus, $d'$ is a braiding. It is obvious that 
$d''_{X, Y}=T_{Y, X}\circ d'_{X, Y}\circ T_{X, Y}^{-1}$, 
and it follows that            
$d''$ is also a braiding, by using Proposition \ref{conjug}. The fact that  
$T$ satisfies (\ref{twine3}) follows immediately by  
using (\ref{conseqcomp1}) and (\ref{conseqcomp2}). 
\end{proof}
\begin{corollary}\label{QT}
Let $H$ be a Hopf algebra, $R\in H\otimes H$ a quasitriangular structure and 
$F\in H\otimes H$ a lazy twist, such that 
\begin{eqnarray}
&&(\Delta \otimes id)(F)=F_{23}R_{12}^{-1}F_{13}R_{12}, \label{compQT1} \\
&&(id\otimes \Delta)(F)=F_{12}R_{23}^{-1}F_{13}R_{23}. \label{compQT2}
\end{eqnarray} 
Then the elements $R'=RF$ and $R''=F_{21}R$ are also quasitriangular 
structures on $H$.
\end{corollary}
\begin{proposition} \label{consteor} 
Let ${\cal C}$ be a monoidal category and $c$, $c'$ braidings on ${\cal C}$.  
Then the inverse braiding $d_{X, Y}:=c_{Y, X}^{-1}$ and the  
laycle $T_{X, Y}=c'_{Y, X}\circ c_{X, Y}$ satisfy the hypotheses of 
Theorem \ref{analog}. Consequently, the family 
$d'_{X, Y}=d_{X, Y}\circ T_{X, Y}=c_{Y, X}^{-1}\circ c'_{Y, X}\circ 
c_{X, Y}$ is a braiding on ${\cal C}$, and the braiding   
$d''$ coincides with the original braiding $c'$.
\end{proposition}
\begin{proof}  
We check (\ref{comp1}):\\[2mm]
${\;\;\;\;\;}$$(id_X\otimes T_{Y, Z})\circ (d_{X, Y}^{-1}\otimes id_Z)
\circ (id_Y\otimes T_{X, Z})\circ (d_{X, Y}\otimes id_Z)$
\begin{eqnarray*}
&=&(id_X\otimes c'_{Z, Y})\circ (id_X\otimes c_{Y, Z})\circ 
(c_{Y, X}\otimes id_Z)\circ (id_Y\otimes c'_{Z, X})\\
&&\circ (id_Y\otimes c_{X, Z})\circ (c_{Y, X}^{-1}\otimes id_Z)\\
&\overset{(\ref{braid1}),\; naturality \;of\;c}{=}& 
(id_X\otimes c'_{Z, Y})\circ (c'_{Z, X}\otimes id_Y)\circ  
(id_Z\otimes c_{Y, X})\circ (c_{Y, Z}\otimes id_X)\\
&&\circ (id_Y\otimes c_{X, Z})\circ (c_{Y, X}^{-1}\otimes id_Z)\\
&\overset{(\ref{braideq})}{=}& 
(id_X\otimes c'_{Z, Y})\circ (c'_{Z, X}\otimes id_Y)\circ  
(c_{X, Z}\otimes id_Y)\circ (id_X\otimes c_{Y, Z})\\
&&\circ (c_{Y, X}\otimes id_Z)\circ (c_{Y, X}^{-1}\otimes id_Z)\\
&\overset{(\ref{braid1}),\; (\ref{braid2})}{=}&
c'_{Z, X\otimes Y}\circ c_{X\otimes Y, Z}\\
&=&T_{X\otimes Y, Z}.
\end{eqnarray*}
The proof of (\ref{comp2}) is similar and left to the reader.
\end{proof}
\begin{remark}
Theorem \ref{analog} together with Proposition \ref{consteor} 
provide an alternative proof of the fact from \cite{brug} that the   
laycle $T_{X, Y}=c'_{Y, X}\circ c_{X, Y}$ is a twine. 
\end{remark}
\begin{remark}
If $({\cal C}, c)$ is a braided monoidal category and we take the inverse 
braiding $d_{X, Y}=c_{Y, X}^{-1}$, then in general $({\cal C}, c)$ and 
$({\cal C}, d)$ are not braided isomorphic. Thus, the braidings $d'$ and 
$d''$ obtained in Theorem \ref{analog} are in general not equivalent 
to the original braiding $d$. 
\end{remark}

Theorem \ref{analog} together with Proposition \ref{consteor} provide 
the following characterization of generalized double braidings: 
\begin{proposition}
Let ${\cal C}$ be a monoidal category and $T$ a laycle on ${\cal C}$. Then 
$T$ is a generalized double braiding if and only if there exists a braiding 
$d$ on ${\cal C}$ such that (\ref{comp1}) and (\ref{comp2}) hold. 
\end{proposition} 


\begin{thebibliography}{99}
\bibitem{bichon}
J. Bichon, G. Carnovale, Lazy cohomology: an analogue of the Schur 
multiplier for arbitrary Hopf algebras, {\sl J. Pure Appl. Algebra} 
{\bf 204} (2006), 627--665.

\bibitem{street}
J. Bichon, R. Street, Militaru's $D$-equation in monoidal categories, 
{\sl Appl. Categ. Structures} {\bf 11} (2003), 337--357.

\bibitem{borcherds1}
R. E. Borcherds, Vertex algebras, In {\sl Topological field theory, 
primitive forms and related topics}, pp. 35--77, Progr. Math. {\bf 160},  
Birkhauser, Boston, 1998.  

\bibitem{borcherds2}
R. E. Borcherds, Quantum vertex algebras, In {\sl Adv. Stud. Pure Math.}  
{\bf 31}, pp. 51--74, Math. Soc. Japan, Tokyo, 2001. 

\bibitem{brug}
A. Brugui\`{e}res, Double braidings, twists and tangle invariants,   
{\sl J. Pure Appl. Algebra} {\bf 204} (2006), 170--194.

\bibitem{carnovale}
G. Carnovale, Some isomorphisms for the Brauer groups of a Hopf algebra, 
{\sl Comm. Algebra} {\bf 29} (2001), 5291--5305.  

\bibitem{c}
G. Carnovale. The Brauer group of modified supergroup algebras, 
{\sl J. Algebra} {\bf 305} (2006), 993--1036.

\bibitem{cc}
G. Carnovale, J. Cuadra, Cocycle twisting of $E(n)$-module algebras 
and applications to the Brauer group, {\sl K-Theory} {\bf 33} (2004), 
251--276.

\bibitem{chen}
H. X. Chen, Skew pairing, cocycle deformations and double 
crossproducts, {\sl Acta Math. Sinica, English Ser.} {\bf 15} (1999), 
225--234.

\bibitem{cp}
J. Cuadra, F. Panaite, Extending lazy 2-cocycles on Hopf algebras and 
lifting projective representations afforded by them, 
{\sl J. Algebra} {\bf 313} (2007), 695--723.

\bibitem{cuqui}
J. Cuntz, D. Quillen, Algebra extensions and nonsingularity, 
{\sl J. Amer. Math. Soc.} {\bf 8} (1995), 251--289.

\bibitem{drinfeld}
V. G. Drinfeld, Quantum groups, In {\sl Proc. Int. Cong. Math. 
(Berkeley, 1986)}, pp. 798--820, Amer. Math. Soc.,  
Providence, RI, 1987.

\bibitem{drin}
V. G. Drinfeld, Quasi-Hopf algebras, {\sl Leningrad Math. J.} 
{\bf 1} (1990), 1419--1457.

\bibitem{durdevich}
M. Durdevich, Generalized braided quantum groups, {\sl Israel J. Math.} 
{\bf 98} (1997), 329--348.

\bibitem{etingof}
P. Etingof, S. Gelaki, Quasisymmetric and unipotent tensor categories, 
preprint arXiv:math.QA/07081487.  

\bibitem{fedosov}
B. V. Fedosov, Analytical formulas for the index of an elliptic 
operator, {\sl Trans. Moscow Math. Soc.} {\bf 30} (1974), 159--240.

\bibitem{gelaki}
S. Gelaki, On pointed ribbon Hopf algebras, {\sl J. Algebra} {\bf 181} 
(1996), 760--786. 

\bibitem{ionescu}
L. M. Ionescu, Cohomology of monoidal categories and non-abelian group 
cohomology, PhD thesis, 2000, available on  
http://www.ilstu.edu.

\bibitem{jlpvo}
P. Jara Mart\'{i}nez, J. L\'{o}pez Pe\~{n}a, F. Panaite, F. Van Oystaeyen,
On iterated twisted tensor products of algebras, preprint 
arXiv:math.QA/0511280, to appear in {\sl Internat. J. Math}.    

\bibitem{joyalstreet}
A. Joyal, R. Street, Braided tensor categories, {\sl Adv. Math.} 
{\bf 102} (1993), 20--78.   

\bibitem{k}
C. Kassel, ''Quantum groups'', {\sl Graduate Texts in Mathematics}
{\bf 155}, Springer Verlag, Berlin, 1995.

\bibitem{leroux}
P. Leroux, On some remarkable operads constructed from Baxter operators, 
preprint arXiv:math.QA/0311214.

\bibitem{liuzhu}
G. Liu, S. Zhu, Almost-triangular Hopf algebras, {\sl Algebr. Represent. 
Theory} {\bf 10} (2007), 555--564.   

\bibitem{lpvo}
J. L\'{o}pez Pe\~{n}a, F. Panaite, F. Van Oystaeyen, 
General twisting of algebras,  
{\sl Adv. Math.} {\bf 212} (2007), 315--337.

\bibitem{maj}
S. Majid, Anyonic quantum groups, In {\sl Spinors, twistors, Clifford 
algebras and quantum deformations}, pp. 327--336, Kluwer Academic 
Publishers, 1993.      

\bibitem{m}
S. Majid, ''Foundations of quantum group theory'', Cambridge Univ.
Press, 1995.

\bibitem{panvo}
F. Panaite, F. Van Oystaeyen, Quasitriangular structures for some pointed 
Hopf algebras of dimension $2^n$, {\sl Comm. Algebra} {\bf 27} (1999), 
4929--4942.

\bibitem{panvan}
F. Panaite, F. Van Oystaeyen,
Clifford-type algebras as cleft extensions for some pointed Hopf 
algebras, {\sl Comm. Algebra} {\bf 28} (2000), 585--600.  

\bibitem{psvo}
F. Panaite, M. D. Staic, F. Van Oystaeyen, On some classes of lazy cocycles 
and categorical structures, {\sl J. Pure Appl. Algebra} {\bf 209} (2007), 
687--701.

\bibitem{pareigis}
B. Pareigis, Symmetric Yetter-Drinfeld categories are trivial, 
{\sl J. Pure Appl. Algebra} {\bf 155} (2001), 91.  

\bibitem{radford}
D. E. Radford, On the antipode of a quasitriangular Hopf algebra, 
{\sl J. Algebra} {\bf 151} (1992), 1--11.

\bibitem{radmin}
D. E. Radford, Minimal quasitriangular Hopf algebras, {\sl J. Algebra} 
{\bf 157} (1993), 285--315.   

\bibitem{sch}
P. Schauenburg, Hopf bimodules, coquasibialgebras, and an exact sequence 
of Kac, {\sl Adv. Math.} {\bf 165} (2002), 194--263.

\bibitem{doru}
M. D. Staic, Pure-braided Hopf algebras and knot invariants, 
{\sl J. Knot Theory Ramifications} {\bf 13} (2004), 385--400.

\bibitem{voz}
F. Van Oystaeyen, Y. Zhang, The Brauer group of a braided monoidal 
category, {\sl J. Algebra} {\bf 202} (1998), 96--128. 


\end{thebibliography}
\end{document}